\makeatletter
\let\ORIlabel\label
\let\ORIrefstepcounter\refstepcounter
\AddToHook{package/hyperref/before}{%
  \let\label\ORIlabel
  \let\refstepcounter\ORIrefstepcounter
}
\makeatother

\documentclass[onefignum,onetabnum]{siamart171218}

\usepackage{amsfonts}
\usepackage{graphicx}
\usepackage[caption=false,font=footnotesize]{subfig}
\usepackage{mathrsfs}
\usepackage{lmodern}
\usepackage{stmaryrd}
\usepackage{fancyvrb}
\usepackage{float}
\usepackage{enumitem}
\ifpdf
  \DeclareGraphicsExtensions{.eps,.pdf,.png,.jpg}
\else
  \DeclareGraphicsExtensions{.eps}
\fi

\newsiamremark{remark}{Remark}
\newsiamremark{hypothesis}{Hypothesis}
\crefname{hypothesis}{Hypothesis}{Hypotheses}
\newsiamthm{claim}{Claim}

\headers{Predictor-Based KS Tracking}{A. Cissé}

\title{Predictor-Based Exponential Tracking of Turbulent Solutions for the Kuramoto--Sivashinsky Equation with Input Delay}

\author{Amadou Cissé \thanks{CRAN, CNRS UMR 7039, University of Lorraine, F-57000 Metz, France 
  (\email{amadou.cisse@univ-lorraine.fr}).}}

\usepackage{amsopn}

\makeatletter
\newcommand*{\addFileDependency}[1]{
  \typeout{(#1)}
  \@addtofilelist{#1}
  \IfFileExists{#1}{}{\typeout{No file #1.}}
}
\makeatother

\ifpdf
\hypersetup{
  pdftitle={Predictor-Based Tracking of KS Trajectories},
  pdfauthor={A. Cissé}
}
\fi

\begin{document}

\maketitle

\begin{abstract}
Uniform exponential tracking is addressed for nonstationary trajectories of the nonlinear Kuramoto--Sivashinsky equation subject to a constant input delay. The reference belongs to a family of complete trajectories contained in the global attractor and need not be stationary, periodic, slowly varying, or generated by a finite-dimensional exosystem. The delayed input is represented by a first-order transport equation coupled with the fourth-order tracking-error dynamics. A predictor--backstepping transformation compensates for the temporal mismatch between command generation and actuation by mapping the augmented closed-loop system into the nominal delay-free error dynamics driven by the outgoing trace of a homogeneous transport subsystem. This subsystem vanishes after one delay interval. Uniform attractor bounds permit the feedback parameters and stability constants to be selected independently of the initial time and the reference trajectory. Global well-posedness and uniform exponential stability are established in the augmented state space. Numerical results show that, for the considered configuration, uncompensated delayed feedback amplifies the tracking error, whereas predictor compensation restores sustained decay. Predictor-consistency, finite-time-extinction, and discretization-refinement diagnostics support the numerical implementation.
\end{abstract}



\section{Introduction}
\label{sec:introduction}

Feedback control of nonlinear partial differential equations exhibiting complex spatiotemporal dynamics becomes particularly demanding when the prescribed behavior is itself nonstationary. For many dissipative equations, however, the long-time dynamics are confined to a compact global attractor of finite fractal dimension. This finite-dimensional asymptotic structure motivates control, observation, and tracking strategies acting through finitely many dynamically relevant components. The Kuramoto--Sivashinsky (KS) equation is a canonical example: destabilizing second-order diffusion, stabilizing fourth-order dissipation, and quadratic transport generate complex dynamics within a finite-dimensional asymptotic structure
\cite{nicolaenko1985global,collet1993global,goodman1994stability,
foias1988inertial,abadie2025topology}.

Most control results for the KS equation concern equilibria or other invariant states. Distributed, finite-dimensional, reduced-order, boundary, and rapid stabilization methods have been developed in \cite{armaou2000feedback,christofides2000global,lhachemi2023local,
cerpa2010null,cerpa2015null,coron2015fredholm,cerpa2017control}.
Comparatively fewer results address prescribed nonstationary trajectories.
Exact controllability toward such trajectories was obtained in
\cite{cerpa2011local,gao2015carleman,gao2020global}, whereas semiglobal feedback stabilization toward time-dependent trajectories through finitely many localized actuators was established in \cite{rodrigues2023feedback}. These results do not incorporate a delay between command generation and its action on the plant.

The finite-dimensional character of dissipative dynamics also motivates the theory of determining modes, nodes, and functionals
\cite{foias1995determining,jolly1990approximate}, continuous data assimilation and nudging \cite{azouani2014continuous,larios2024nonlinear}, and robust or adaptive observers for nonlinear KS dynamics \cite{cisse2025hinfinity,cisse2026adaptive}. Such approaches primarily seek state reconstruction or synchronization through partial observations. The objective pursued here is a control problem: a finite-dimensional signal is physically injected into the plant to force its state to track a prescribed trajectory.

Uniform exponential tracking of complete KS trajectories contained in the global attractor was recently established through finite-dimensional spectral feedback in \cite{cisse2026uniform}. The controlled-mode dimension, feedback gain, and convergence estimates are uniform over the entire reference family, while robustness bounds yield corresponding conclusions for long-time observables. The feedback is nevertheless assumed to act instantaneously.
Actuator, transmission, and computational delays instead cause the command reaching the plant to have been generated for an earlier tracking configuration.
For a nonstationary reference, such a delay-induced misalignment can destroy the dissipative action of the nominal feedback.

Delayed control has also been studied for KS-type equations. A finite-dimensional predictor acting on the unstable modes yields exponential stabilization of the linearized KS equation subject to delayed boundary control \cite{guzman2019stabilization}. Well-posedness, dissipativity, and decay have been investigated for nonlinear KS and related equations containing boundary delays or memory terms
\cite{chentouf2021delay,chentouf2022infinite,chentouf2022finite,
chentouf2023singular,chentouf2025well}. These contributions concern stabilization toward an equilibrium or the analysis of equations containing delayed terms.
They do not address delayed-feedback tracking of arbitrary complete trajectories of the nonlinear periodic KS equation.

Predictor feedback compensates for an input delay by representing the actuator as a transport equation and generating the command according to the state predicted at the instant when that command becomes effective. This principle has led to predictor and backstepping designs for finite-dimensional systems, reaction--diffusion equations, and broader classes of infinite-dimensional systems \cite{krstic2009delay,krstic2009reaction,lhachemi2020iss,
lhachemi2021robustness,lhachemi2022predictor,vazquez2026survey}, including delay-adaptive control of reaction--advection--diffusion equations subject to delayed distributed inputs \cite{9528941}. Existing PDE designs are primarily devoted to equilibrium or set-point stabilization and do not directly resolve the present tracking problem. The KS tracking-error equation is both nonlinear and nonautonomous; its reference is an infinite-dimensional trajectory rather than the output of a finite-dimensional exosystem; and the predictor transformation must satisfy estimates uniform over an entire family of complete trajectories.

The central question is therefore whether delayed finite-dimensional actuation can preserve uniform exponential tracking of admissible nonstationary KS trajectories. To the best of the authors' knowledge, no available result combines trajectory-uniform tracking, predictor compensation of a distributed-input delay, and global well-posedness of the resulting nonlinear KS--transport closed loop.

A constant and exactly known input delay is considered. The reference may be any complete trajectory contained in the global attractor and need not be stationary, periodic, slowly varying, or generated by a finite-dimensional exosystem. As is standard in prescribed-trajectory predictor control, the reference is assumed to be available over the prediction horizon. Subject to this preview information, the controller is causal with respect to the plant state and the actuator history: it uses no future value of the controlled state.
The objective is uniform exponential tracking through delayed
finite-dimensional spectral actuation, with controller parameters and stability estimates independent of both the initial time and the selected reference.

The analysis couples the nonlinear reference-dependent tracking dynamics with the hyperbolic transport representation of the actuator delay. It must preserve the coercivity supplied by the fourth-order dissipation while controlling the quadratic convection, establish mutually inverse transformations between the original and target coordinates, and obtain bounds uniform over the reference family. The delayed problem therefore cannot be reduced to either the nominal delay-free tracking result or a linear predictor construction.

The main contributions are fourfold. First, a nonlinear predictor is constructed over the delay horizon, and its terminal state determines a finite-dimensional spectral feedback law. Second, a predictor--backstepping transformation maps the augmented closed loop into a cascade formed by the nominal stabilized tracking dynamics and a homogeneous transport subsystem. The latter acts on the tracking error through its outgoing trace and vanishes exactly after one delay interval.
Third, the direct and inverse transformations are established on the augmented state space, which yields global well-posedness of the nonlinear KS--transport closed loop for mild initial data. Fourth, uniform exponential stability is proved in the augmented state space. The controlled-mode dimension, feedback gain, transformation bounds, and stability constants are independent of the initial time and the selected complete reference trajectory.

These results extend uniform finite-dimensional tracking of nonstationary KS trajectories to delayed actuation. The predictor neither removes nor shortens the physical delay. It restores the temporal alignment between command generation and the future state on which the command acts. Numerical experiments compare the uncontrolled system, uncompensated delayed feedback, and predictor compensation. They illustrate the loss and recovery of tracking, the terminal predictor relation, and the finite-time evacuation of the transformed actuator state. The observed behavior is consistent with the mechanisms established by the stability analysis.

Section~\ref{sec:problem} formulates the delayed tracking problem and introduces the transport representation of the input delay.
Section~\ref{sec:predictor-design} constructs the predictor transformation, establishes the equivalence between the original and target systems, and proves global well-posedness and uniform exponential stability.
Section~\ref{sec:numerical-illustration} presents the numerical implementation and results. Section~\ref{sec:conclusion} concludes the paper.

\subsection*{Notation}
\label{subsec:notation}

All function spaces are real. Let \(\mathbb T_L=\mathbb R/(L\mathbb Z)\) be the one-dimensional torus of length \(L>0\), and let \(H_{\mathrm{per}}^k(\mathbb T_L)\) denote the periodic Sobolev space of order \(k\). The zero-mean spaces used throughout are
\[
H:=\dot L^2(\mathbb T_L)
 =\left\{\varphi\in L^2(\mathbb T_L):
 \int_0^L\varphi(x)\,dx=0\right\},
\quad
V:=\dot H_{\mathrm{per}}^2(\mathbb T_L)
 =H_{\mathrm{per}}^2(\mathbb T_L)\cap H.
\]
The space \(H\) is endowed with
\(
\langle\varphi,\psi\rangle_H
 :=\int_0^L\varphi(x)\psi(x)\,dx,
\)
and \(V\) with the equivalent norm
\(\|\varphi\|_V^2=\|\varphi\|_{L^2}^2+\|\varphi_{xx}\|_{L^2}^2\).
Identifying \(H\) with its dual gives the Gelfand triple
\(V\hookrightarrow H\equiv H'\hookrightarrow V'\), where
\(V'\simeq\dot H_{\mathrm{per}}^{-2}(\mathbb T_L)\). The duality pairing between \(V'\) and \(V\) is denoted by
\(\langle\cdot,\cdot\rangle_{V',V}\), and
\(
\|f\|_{V'}
 :=\sup_{\varphi\in V\setminus\{0\}}
 \frac{|\langle f,\varphi\rangle_{V',V}|}{\|\varphi\|_V}.
\)

For a Banach space \(X\), \(\|\cdot\|_X\), \(C(I;X)\), and \(L^p(I;X)\) have their usual meanings. The space of bounded linear operators from \(X\) to \(Y\) is denoted by \(\mathcal L(X,Y)\), with \(\mathcal L(X):=\mathcal L(X,X)\). For Hilbert spaces \(X\) and \(Y\), \(T^*\) denotes the adjoint of \(T\in\mathcal L(X,Y)\).
The domain, kernel, and range of \(T\) are denoted by \(D(T)\), \(\ker(T)\), and \(\operatorname{Ran}(T)\), respectively.
For a bounded self-adjoint operator \(P\in\mathcal L(X)\),
\(P\succeq\alpha I_X\) means \(\langle Px,x\rangle_X\geq\alpha\|x\|_X^2\) for every \(x\in X\).

The orthogonal projection onto the controlled Fourier subspace \(H_N\) is denoted by \(P_N\), and \(Q_N:=I_H-P_N\). For the delay state, set \(\mathcal Z:=L^2((0,D);\mathbb R^m), \;
\mathcal X:=H\times\mathcal Z,\)
\(
\|(w,z)\|_{\mathcal X}^2 :=\|w\|_H^2+\int_0^D\|z(s)\|_{\mathbb R^m}^2\,ds.
\)
Subscripts denote partial derivatives, while \(\dot w\) denotes the time derivative of an \(H\)-valued function. Spatial dependence is omitted when unambiguous, equalities in \(L^2\)-spaces hold almost everywhere, and \(C>0\) denotes a generic constant that may vary between occurrences.

\section{Problem Formulation}
\label{sec:problem}

The delayed tracking problem for the Kuramoto--Sivashinsky equation is formulated by specifying the controlled plant, the finite-dimensional actuation mechanism, the admissible reference trajectories, and the tracking objective.
A transport PDE is then introduced to represent the input delay, leading to an augmented system coupling the tracking error with the actuator delay-line state.

\subsection{Delayed Kuramoto--Sivashinsky Model and Actuation}
\label{subsec:delayed-KS-model}

Let \(\tau\geq0\) be an arbitrary initial time. Consider the
Kuramoto--Sivashinsky equation with a delayed finite-dimensional
distributed control input:
\begin{equation}
\begin{aligned}
u_t(t,x)
+u_{xxxx}(t,x)
+u_{xx}(t,x)
+u(t,x)u_x(t,x)
&=
\sum_{j=1}^{m} b_j(x)U_j(t-D),
\\
(t,x)&\in(\tau,\infty)\times\mathbb T_L.
\end{aligned}
\label{eq:delayed-KS}
\end{equation}
The periodic boundary conditions are
\begin{equation}
\partial_x^\ell u(t,0)
=
\partial_x^\ell u(t,L),
\qquad
\ell=0,1,2,3,
\qquad
t\geq\tau,
\label{eq:KS-periodic-BC}
\end{equation}
and the initial condition is
\begin{equation}
u(\tau,x)=u_\tau(x),
\qquad
x\in\mathbb T_L.
\label{eq:KS-initial-condition}
\end{equation}
Here,
\[
U(t)
=
\operatorname{col}
\bigl(U_1(t),\ldots,U_m(t)\bigr)
\in\mathbb R^m
\]
is the control signal, \(b_j\) is the spatial profile of the \(j\)-th actuation channel, and \(D>0\) is a known constant input delay. The control history at time \(\tau\) is prescribed by
\begin{equation}
U(\theta)=U_\tau(\theta),
\qquad
\theta\in[\tau-D,\tau),
\qquad
U_\tau\in
L^2\bigl((\tau-D,\tau);\mathbb R^m\bigr).
\label{eq:control-history}
\end{equation}

Define
\begin{equation}
\mathcal A\varphi
=
-\partial_x^4\varphi-\partial_x^2\varphi,
\qquad
D(\mathcal A)
=
H_{\mathrm{per}}^4(\mathbb T_L)\cap H,
\label{eq:KS-linear-operator}
\end{equation}
and
\begin{equation}
\mathcal N(\varphi)
=
-\varphi\partial_x\varphi
=
-\frac{1}{2}\partial_x(\varphi^2),
\qquad
\varphi\in H_{\mathrm{per}}^1(\mathbb T_L)\cap H.
\label{eq:KS-nonlinearity}
\end{equation}
Since \(H_{\mathrm{per}}^1(\mathbb T_L)\) is continuously embedded into \(L^\infty(\mathbb T_L)\), the mapping
\[
\mathcal N:
H_{\mathrm{per}}^1(\mathbb T_L)\cap H
\longrightarrow H
\]
is well defined. Equation~\eqref{eq:delayed-KS} can then be written as
\begin{equation}
\dot u(t)
=
\mathcal Au(t)
+
\mathcal N\bigl(u(t)\bigr)
+
BU(t-D),
\qquad
u(\tau)=u_\tau,
\label{eq:abstract-delayed-KS}
\end{equation}
where \(B\in\mathcal L(\mathbb R^m,H)\) is defined by
\begin{equation}
B\xi
=
\sum_{j=1}^{m}b_j\xi_j,
\qquad
\xi
=
\operatorname{col}(\xi_1,\ldots,\xi_m)
\in\mathbb R^m.
\label{eq:control-operator}
\end{equation}

The control acts on the first \(N\) nonzero Fourier wavenumbers,
including their cosine and sine components. More precisely, let
\(N\in\mathbb N\), set \(m=2N\), and define
\begin{equation}
\begin{aligned}
b_{2k-1}(x)
&=
\sqrt{\frac{2}{L}}
\cos\left(\frac{2\pi kx}{L}\right),
\\
b_{2k}(x)
&=
\sqrt{\frac{2}{L}}
\sin\left(\frac{2\pi kx}{L}\right),
\qquad
k=1,\ldots,N.
\end{aligned}
\label{eq:Fourier-actuators}
\end{equation}
These functions form an orthonormal family in \(H\). Hence,
\begin{equation}
B^\ast\varphi
=
\operatorname{col}
\bigl(
\langle\varphi,b_1\rangle_H,\ldots,
\langle\varphi,b_{2N}\rangle_H
\bigr),
\qquad
\varphi\in H,
\label{eq:adjoint-control-operator}
\end{equation}
and
\begin{equation}
B^\ast B=I_{\mathbb R^{2N}},
\qquad
\|B\xi\|_H=\|\xi\|_{\mathbb R^{2N}}.
\label{eq:control-operator-isometry}
\end{equation}
Thus, \(B\) is an isometric embedding whose range is
\begin{equation}
H_N
=
\operatorname{Ran}(B)
=
\operatorname{span}\{b_1,\ldots,b_{2N}\}.
\label{eq:controlled-subspace}
\end{equation}
The orthogonal projection onto \(H_N\) is consequently
\begin{equation}
P_N=BB^\ast.
\label{eq:controlled-projection}
\end{equation}

The actuator profiles are eigenfunctions of \(\mathcal A\). Indeed, setting
\begin{equation}
q_k=\frac{2\pi k}{L},
\qquad
k\geq1,
\label{eq:Fourier-wavenumbers}
\end{equation}
gives
\begin{equation}
\mathcal A b_{2k-1}
=
\lambda_k b_{2k-1},
\qquad
\mathcal A b_{2k}
=
\lambda_k b_{2k},
\qquad
\lambda_k=q_k^2-q_k^4.
\label{eq:KS-eigenvalues}
\end{equation}
Since \(\lambda_k\geq0\) if and only if \(q_k\leq1\), define
\begin{equation}
N_u
=
\max
\left(
\left\{
k\in\mathbb N:q_k\leq1
\right\}
\cup\{0\}
\right).
\label{eq:number-nonnegative-wavenumbers}
\end{equation}
The number of controlled wavenumbers is required to satisfy
\begin{equation}
N\geq N_u.
\label{eq:minimal-number-controlled-modes}
\end{equation}
Consequently, \(H_N\) contains all eigenspaces of \(\mathcal A\)
associated with nonnegative eigenvalues. This condition ensures the minimal spectral coverage of the nondecaying modes of the uncontrolled linear KS operator. It does not by itself imply uniform stabilizability of the reference-dependent linearized tracking dynamics, whose variable coefficient couples the Fourier modes. Accordingly, \(N\) will be chosen sufficiently large to satisfy the uniform stabilizability condition introduced in
Section~\ref{sec:predictor-design}.

Finally, the zero-mean space \(H\) is invariant under the controlled dynamics. Indeed, periodicity and the zero spatial mean of the actuator profiles yield
\[
\int_0^L
\left(
u_{xxxx}+u_{xx}+uu_x
\right)\,dx
=
0,
\qquad
\int_0^L b_j(x)\,dx=0.
\]
Therefore, every sufficiently regular solution satisfies
\begin{equation}
\frac{d}{dt}\int_0^L u(t,x)\,dx=0,
\qquad
t\geq\tau.
\label{eq:mean-conservation}
\end{equation}
For mild solutions, the same property follows from the invariance of \(H\) under the semigroup generated by \(\mathcal A\), together with
\[
\mathcal N(\varphi)\in H,
\qquad
\operatorname{Ran}(B)\subset H.
\]
Hence, \(u_\tau\in H\) implies that \(u(t)\in H\) throughout the maximal interval of existence.

\subsection{Transport Representation and Tracking Dynamics}
\label{subsec:augmented-tracking-dynamics}

Let \(\tau\geq0\) be the initial time introduced in Subsection~\ref{subsec:delayed-KS-model}. The input delay is represented by the transport state
\begin{equation}
z(t,s)
=
U(t+s-D),
\qquad
(t,s)\in[\tau,\infty)\times[0,D],
\label{eq:transport-state-definition}
\end{equation}
where \(U\) denotes the control signal extended to
\([\tau-D,\tau)\) by the prescribed history \(U_\tau\). Formally,
\(z\) satisfies
\begin{equation}
\begin{cases}
z_t(t,s)=z_s(t,s),
&
(t,s)\in(\tau,\infty)\times(0,D),
\\[1mm]
z(t,D)=U(t),
&
t>\tau,
\\[1mm]
z(\tau,s)=z_\tau(s),
&
s\in(0,D),
\end{cases}
\label{eq:transport-delay-system}
\end{equation}
where
\begin{equation}
z_\tau(s)
=
U_\tau(\tau+s-D),
\qquad
\text{for a.e. }s\in(0,D).
\label{eq:initial-transport-state}
\end{equation}
Thus,
\[
z_\tau\in
\mathcal Z
:=
L^2\bigl((0,D);\mathbb R^m\bigr).
\]

Set
\[
\theta=t-\tau.
\]
The characteristic representation of
\eqref{eq:transport-delay-system} is
\begin{equation}
z(t,s)
=
\begin{cases}
z_\tau(s+\theta),
&
s+\theta<D,
\\[1mm]
U(t+s-D),
&
s+\theta\geq D,
\end{cases}
\label{eq:transport-characteristic-representation}
\end{equation}
for almost every \(s\in(0,D)\). In particular, the outgoing trace
satisfies
\begin{equation}
z(t,0)
=
\begin{cases}
z_\tau(\theta)
=
U_\tau(t-D),
&
0<\theta<D,
\\[1mm]
U(t-D),
&
\theta\geq D,
\end{cases}
\label{eq:delayed-input-trace}
\end{equation}
for almost every \(t>\tau\). Equivalently,
\[
z(t,0)=U(t-D)
\]
when \(U\) denotes the extended control signal.

The delayed Kuramoto--Sivashinsky equation can consequently be written as
\begin{equation}
\dot u(t)
=
\mathcal Au(t)
+
\mathcal N\bigl(u(t)\bigr)
+
Bz(t,0),
\qquad
u(\tau)=u_\tau,
\label{eq:KS-transport-representation}
\end{equation}
coupled with~\eqref{eq:transport-delay-system}, with
\[
\bigl(u(t),z(t,\cdot)\bigr)
\in
\mathcal X
:=
H\times\mathcal Z.
\]

Let \(\mathscr A\) denote the global attractor of the uncontrolled
Kuramoto--Sivashinsky equation on \(H\). Since \(\mathscr A\) is compact
in \(H_{\mathrm{per}}^2(\mathbb T_L)\cap H\), the constant
\begin{equation}
M_{\mathscr A}
:=
\sup_{\varphi\in\mathscr A}
\|\varphi\|_{H_{\mathrm{per}}^2}
<
\infty
\label{eq:attractor-H2-bound}
\end{equation}
is well defined.

Consider an arbitrary complete reference trajectory
\[
u_{\mathrm{ref}}:\mathbb R\longrightarrow\mathscr A
\]
satisfying the uncontrolled Kuramoto--Sivashinsky equation
\begin{equation}
\dot u_{\mathrm{ref}}(t)
=
\mathcal Au_{\mathrm{ref}}(t)
+
\mathcal N\bigl(u_{\mathrm{ref}}(t)\bigr).
\label{eq:reference-dynamics}
\end{equation}
This identity is understood in the mild sense and, whenever the
trajectory has the corresponding regularity, in \(H\) for almost every
\(t\in\mathbb R\). Since
\(u_{\mathrm{ref}}(t)\in\mathscr A\) for every \(t\in\mathbb R\),
\begin{equation}
\sup_{t\in\mathbb R}
\|u_{\mathrm{ref}}(t)\|_{H_{\mathrm{per}}^2}
\leq
M_{\mathscr A}.
\label{eq:reference-uniform-H2-bound}
\end{equation}
In particular, for every \(D>0\),
\begin{equation}
\sup_{t\in\mathbb R}
\sup_{s\in[0,D]}
\|u_{\mathrm{ref}}(t+s)\|_{H_{\mathrm{per}}^2}
\leq
M_{\mathscr A}.
\label{eq:reference-segment-uniform-bound}
\end{equation}
This bound is uniform over all admissible complete reference
trajectories.

Define the tracking error by
\begin{equation}
w(t)
=
u(t)-u_{\mathrm{ref}}(t),
\qquad
w_\tau
=
u_\tau-u_{\mathrm{ref}}(\tau).
\label{eq:tracking-error-definition}
\end{equation}
Subtracting \eqref{eq:reference-dynamics} from
\eqref{eq:KS-transport-representation} gives
\begin{equation}
\dot w(t)
=
\mathcal Aw(t)
+
\mathcal L(t)w(t)
+
\mathcal R\bigl(w(t)\bigr)
+
Bz(t,0),
\qquad
w(\tau)=w_\tau,
\label{eq:abstract-tracking-error}
\end{equation}
where
\begin{equation}
\begin{aligned}
\mathcal L(t):
H_{\mathrm{per}}^1(\mathbb T_L)\cap H
&\longrightarrow H,
&
\mathcal L(t)\varphi
&=
-\partial_x
\bigl(u_{\mathrm{ref}}(t)\varphi\bigr),
\\
\mathcal R:
H_{\mathrm{per}}^1(\mathbb T_L)\cap H
&\longrightarrow H,
&
\mathcal R(\varphi)
&=
-\frac{1}{2}\partial_x(\varphi^2).
\end{aligned}
\label{eq:tracking-error-operators}
\end{equation}
Indeed,
\[
\mathcal N\bigl(u_{\mathrm{ref}}+w\bigr)
-
\mathcal N\bigl(u_{\mathrm{ref}}\bigr)
=
-\partial_x\bigl(u_{\mathrm{ref}}w\bigr)
-\frac{1}{2}\partial_x(w^2)
=
\mathcal L(t)w+\mathcal R(w).
\]
The embedding
\[
H_{\mathrm{per}}^1(\mathbb T_L)
\hookrightarrow
L^\infty(\mathbb T_L),
\]
the regularity
\[
u_{\mathrm{ref}}(t)
\in
H_{\mathrm{per}}^2(\mathbb T_L),
\]
and the vanishing spatial mean of periodic derivatives ensure that
\(\mathcal L(t)\varphi\) and \(\mathcal R(\varphi)\) belong to \(H\).

The tracking error inherits the periodic boundary conditions
\begin{equation}
\partial_x^\ell w(t,0)
=
\partial_x^\ell w(t,L),
\qquad
\ell=0,1,2,3,
\qquad
t\geq\tau.
\label{eq:tracking-error-periodic-BC}
\end{equation}
For mild solutions, the corresponding variation-of-constants formula is
\begin{equation}
\begin{aligned}
w(t)
&=
e^{(t-\tau)\mathcal A}w_\tau
\\
&\quad
+
\int_\tau^t
e^{(t-r)\mathcal A}
\left[
\mathcal L(r)w(r)
+
\mathcal R\bigl(w(r)\bigr)
+
Bz(r,0)
\right]\,dr,
\qquad
t\geq\tau,
\end{aligned}
\label{eq:tracking-error-mild-form}
\end{equation}
whenever the terms in the integral have the required regularity.

The augmented tracking dynamics are therefore
\begin{equation}
\begin{cases}
\dot w(t)
=
\mathcal Aw(t)
+
\mathcal L(t)w(t)
+
\mathcal R\bigl(w(t)\bigr)
+
Bz(t,0),
&
t>\tau,
\\[1mm]
z_t(t,s)=z_s(t,s),
&
(t,s)\in(\tau,\infty)\times(0,D),
\\[1mm]
z(t,D)=U(t),
&
t>\tau,
\\[1mm]
w(\tau)=w_\tau,
\qquad
z(\tau,\cdot)=z_\tau.
\end{cases}
\label{eq:augmented-tracking-system}
\end{equation}
The corresponding augmented state satisfies
\[
\bigl(w(t),z(t,\cdot)\bigr)
\in
\mathcal X,
\qquad
t\geq\tau.
\]

\subsection{Control Objective}
\label{subsec:control-objective}

The delayed control input is sought in the form
\begin{equation}
U(t)
=
\mathfrak K
\left(
t,
w(t),
z(t,\cdot),
u_{\mathrm{ref}}^{\,t}
\right),
\qquad
t\geq \tau,
\label{eq:general-feedback-law}
\end{equation}
where the reference segment \(u_{\mathrm{ref}}^{\,t}\) is defined by
\begin{equation}
\begin{aligned}
u_{\mathrm{ref}}^{\,t}:
[0,D]
&\longrightarrow
H_{\mathrm{per}}^2(\mathbb T_L)\cap H,
\\
u_{\mathrm{ref}}^{\,t}(s)
&=
u_{\mathrm{ref}}(t+s).
\end{aligned}
\label{eq:available-reference-segment}
\end{equation}

The segment \(u_{\mathrm{ref}}^{\,t}\) provides a preview of the
reference over \([t,t+D]\). Its availability at time \(t\) follows from the prescription of \(u_{\mathrm{ref}}\) on the whole real line. By contrast, the plant-dependent arguments of the feedback are restricted to the current augmented tracking state
\[
\bigl(w(t),z(t,\cdot)\bigr).
\]
In particular, the feedback law involves no future value of the plant state, tracking error, or delay-line state.

The feedback must satisfy the origin-consistency condition
\begin{equation}
\mathfrak K
\left(
t,
0,
0,
u_{\mathrm{ref}}^{\,t}
\right)
=
0,
\qquad
t\geq\tau,
\label{eq:feedback-origin-consistency}
\end{equation}
for every prescribed complete reference trajectory contained in
\(\mathscr A\). Consequently, \(U(t)=0\) along the zero solution of the closed-loop augmented tracking system.

The control objective is to construct \(\mathfrak K\) and constants
\[
M\geq1,
\qquad
\omega>0,
\]
such that the origin of the closed-loop augmented tracking
system~\eqref{eq:augmented-tracking-system} is globally exponentially stable, uniformly with respect to the initial time and the prescribed complete reference trajectory.

More precisely, let \(\tau\geq0\) be arbitrary and consider initial data
\[
w(\tau)=w_\tau,
\qquad
z(\tau,\cdot)=z_\tau
\]
satisfying the delay-history compatibility condition
\begin{equation}
z_\tau(s)
=
U_\tau(s-D),
\qquad
\text{for almost every }s\in(0,D).
\label{eq:initial-delay-history-consistency}
\end{equation}
The resulting closed-loop augmented system must admit a unique global mild solution
\[
\bigl(w(t),z(t,\cdot)\bigr)\in\mathcal X,
\qquad
t\geq\tau,
\]
satisfying
\begin{equation}
\left\|
\bigl(w(t),z(t,\cdot)\bigr)
\right\|_{\mathcal X}
\leq
M e^{-\omega(t-\tau)}
\left\|
\bigl(w_\tau,z_\tau\bigr)
\right\|_{\mathcal X},
\qquad
t\geq\tau.
\label{eq:uniform-global-exponential-stability-objective}
\end{equation}
The constants \(M\) and \(\omega\) may depend on the system parameters, the delay \(D\), the actuator configuration, and the uniform attractor bound \(M_{\mathscr A}\), but remain independent of \(\tau\) and the selected complete reference trajectory.

Since
\[
\|w(t)\|_H
\leq
\left\|
\bigl(w(t),z(t,\cdot)\bigr)
\right\|_{\mathcal X},
\]
estimate~\eqref{eq:uniform-global-exponential-stability-objective}
implies
\begin{equation}
\|u(t)-u_{\mathrm{ref}}(t)\|_H
\leq
M e^{-\omega(t-\tau)}
\left\|
\bigl(w_\tau,z_\tau\bigr)
\right\|_{\mathcal X},
\qquad
t\geq\tau.
\label{eq:uniform-exponential-tracking-objective}
\end{equation}
Hence, uniform global exponential stability of the augmented tracking origin yields global exponential tracking, with a decay rate uniform over all initial times and all prescribed complete reference trajectories contained in \(\mathscr A\).

The feedback construction proceeds in two stages. A nominal feedback is first designed for the delay-free reference-dependent tracking dynamics.
A nonlinear predictor transformation is then introduced to compensate for the actuator delay.

\section{Predictor--Backstepping Design}
\label{sec:predictor-design}

The feedback law for the delayed tracking
system~\eqref{eq:augmented-tracking-system} is constructed below. A finite-dimensional feedback is first designed for the delay-free tracking dynamics. The resulting nominal system defines a nonlinear predictor over the delay horizon. A backstepping transformation will then convert the delayed closed-loop system into an exponentially stable cascade.

\subsection{Nominal Feedback and Uniform Stability}
\label{subsec:nominal-feedback}

For \(k\geq1\), set
\begin{equation}
q_k=\frac{2\pi k}{L},
\qquad
a_k=q_k^4-q_k^2.
\label{eq:positive-KS-spectral-coefficients}
\end{equation}
By~\eqref{eq:KS-eigenvalues},
\[
\mathcal A b_{2k-1}=-a_kb_{2k-1},
\qquad
\mathcal A b_{2k}=-a_kb_{2k},
\]
and \(a_k\to+\infty\) as \(k\to\infty\).

The attractor bound~\eqref{eq:reference-uniform-H2-bound} and the
embedding
\(H_{\mathrm{per}}^1(\mathbb T_L)\hookrightarrow
L^\infty(\mathbb T_L)\) imply
\begin{equation}
\frac12
\|\partial_xu_{\mathrm{ref}}(t)\|_{L^\infty}
\leq
C_{\mathscr A}
:=
\frac12 C_{\mathrm{emb}}M_{\mathscr A},
\qquad
t\in\mathbb R,
\label{eq:reference-derivative-uniform-bound}
\end{equation}
uniformly over all complete reference trajectories in \(\mathscr A\).

Choose \(N\geq N_u\) and \(\mu>0\) such that
\begin{equation}
\underline a_{N,\mu}
:=
\min
\left\{
\min_{1\leq k\leq N}(a_k+\mu),
\inf_{k\geq N+1}a_k
\right\}
>
C_{\mathscr A},
\qquad
\eta
:=
\underline a_{N,\mu}-C_{\mathscr A}>0.
\label{eq:nominal-spectral-condition}
\end{equation}
The divergence \(a_k\to+\infty\) guarantees the existence of such a choice:
one first selects \(N\) so that the uncontrolled tail lies above
\(C_{\mathscr A}\), and then chooses \(\mu\) so that the same property holds for the first \(N\) modes.

Define
\begin{equation}
\kappa(\varphi)
=
-\mu B^\ast\varphi,
\qquad
\varphi\in H.
\label{eq:nominal-feedback-map}
\end{equation}
Since \(BB^\ast=P_N\),
\begin{equation}
B\kappa(\varphi)=-\mu P_N\varphi,
\qquad
\|\kappa(\varphi)\|_{\mathbb R^m}
\leq
\mu\|B^\ast\|_{\mathcal L(H,\mathbb R^m)}
\|\varphi\|_H.
\label{eq:nominal-feedback-properties}
\end{equation}

For \(\tau\geq0\), consider
\begin{equation}
\dot q
=
\mathcal Aq
+
\mathcal L(t)q
+
\mathcal R(q)
-
\mu P_Nq
+
d,
\qquad
q(\tau)=q_\tau,
\label{eq:nominal-closed-loop-system}
\end{equation}
where
\[
q_\tau\in H,
\qquad
d\in L_{\mathrm{loc}}^2([\tau,\infty);H).
\]
Let
\[
V=H_{\mathrm{per}}^2(\mathbb T_L)\cap H,
\qquad
V\hookrightarrow H\hookrightarrow V'.
\]
Equation~\eqref{eq:nominal-closed-loop-system} is understood in \(V'\) almost everywhere and in the mild sense in \(H\).

\begin{proposition}
\label{prop:nominal-uniform-ISS}
Let \(N\) and \(\mu\) satisfy
\eqref{eq:nominal-spectral-condition}. Then
\eqref{eq:nominal-closed-loop-system} admits a unique global solution
\[
q\in
C([\tau,\infty);H)
\cap
L_{\mathrm{loc}}^2([\tau,\infty);V),
\qquad
\dot q\in
L_{\mathrm{loc}}^2([\tau,\infty);V').
\]
Moreover,
\begin{equation}
\|q(t)\|_H^2
\leq
e^{-\eta(t-\tau)}\|q_\tau\|_H^2
+
\frac1\eta
\int_\tau^t
e^{-\eta(t-\sigma)}
\|d(\sigma)\|_H^2\,d\sigma,
\qquad
t\geq\tau.
\label{eq:nominal-ISS-estimate}
\end{equation}
The estimate holds uniformly with respect to \(\tau\) and the selected complete reference trajectory.
\end{proposition}

\begin{proof}
Let \(q_n\) be the Fourier--Galerkin approximation initialized by
\(q_n(\tau)=\Pi_nq_\tau\). Periodicity gives
\[
\langle\mathcal R(q_n),q_n\rangle_H=0
\]
and
\[
\langle\mathcal L(t)q_n,q_n\rangle_H
=
-\frac12
\int_0^L
\partial_xu_{\mathrm{ref}}(t,x)q_n^2(x)\,dx
\leq
C_{\mathscr A}\|q_n\|_H^2.
\]
The Fourier representation and
\eqref{eq:nominal-spectral-condition} yield
\[
-\langle\mathcal Aq_n,q_n\rangle_H
+
\mu\|P_Nq_n\|_H^2
\geq
\underline a_{N,\mu}\|q_n\|_H^2.
\]
Consequently,
\[
\frac12\frac{d}{dt}\|q_n\|_H^2
+
\eta\|q_n\|_H^2
\leq
\langle d,q_n\rangle_H.
\]
Young's inequality gives
\begin{equation}
\frac{d}{dt}\|q_n\|_H^2
+
\eta\|q_n\|_H^2
\leq
\frac1\eta\|d\|_H^2,
\label{eq:nominal-energy-differential-inequality}
\end{equation}
and integration over \([\tau,t]\) yields
\begin{equation}
\|q_n(t)\|_H^2
\leq
e^{-\eta(t-\tau)}
\|\Pi_nq_\tau\|_H^2
+
\frac1\eta
\int_\tau^t
e^{-\eta(t-\sigma)}
\|d(\sigma)\|_H^2\,d\sigma.
\label{eq:nominal-Galerkin-ISS-estimate}
\end{equation}

For \(\varphi\in V\),
\[
-\langle\mathcal A\varphi,\varphi\rangle_H
=
\|\varphi_{xx}\|_{L^2}^2
-
\|\varphi_x\|_{L^2}^2
\geq
\frac12\|\varphi_{xx}\|_{L^2}^2
-
C_L\|\varphi\|_H^2.
\]
Together with~\eqref{eq:nominal-Galerkin-ISS-estimate}, this provides bounds independent of \(n\) in
\[
L^\infty((\tau,T);H)\cap L^2((\tau,T);V)
\]
for every \(T>\tau\). Moreover,
\[
\|\mathcal L(t)q_n\|_{V'}
\leq
CM_{\mathscr A}\|q_n\|_H,
\qquad
\|\mathcal R(q_n)\|_{V'}
\leq
C\|q_n\|_H^2,
\]
so that \((\dot q_n)\) is bounded in \(L^2((\tau,T);V')\).
Compactness and the Aubin--Lions lemma permit passage to the quadratic term and give
\[
q\in C([\tau,T];H)\cap L^2((\tau,T);V),
\qquad
\dot q\in L^2((\tau,T);V').
\]
Weak lower semicontinuity gives
\eqref{eq:nominal-ISS-estimate}.

For two solutions \(q_1,q_2\), the difference satisfies
\[
\frac{d}{dt}\|q_1-q_2\|_H^2
\leq
C\bigl(
1+\|q_1\|_V^2+\|q_2\|_V^2
\bigr)
\|q_1-q_2\|_H^2,
\]
and uniqueness follows from Gronwall's inequality. The standard
continuation criterion for the semilinear equation, together with the uniform \(H\)-bound and the local \(L^2(V)\)-bound above, excludes a finite maximal existence time. Hence the solution is global. All constants depend only on \(M_{\mathscr A}\), the domain, and the fixed design parameters.
\end{proof}

For \(d=0\), Proposition~\ref{prop:nominal-uniform-ISS} gives
\begin{equation}
\|q(t)\|_H
\leq
e^{-\frac{\eta}{2}(t-\tau)}
\|q_\tau\|_H,
\qquad
t\geq\tau.
\label{eq:nominal-uniform-exponential-stability}
\end{equation}
Thus, the nominal origin is globally exponentially stable, uniformly
with respect to the initial time and the prescribed reference trajectory.

\subsection{Nonlinear Predictor and Delay Compensation}
\label{subsec:nonlinear-predictor}

For \(t\geq\tau\), define
\[
p(t,\cdot):[0,D]\longrightarrow H
\]
as the solution of
\begin{equation}
\begin{cases}
\displaystyle
\partial_sp(t,s)
=
\mathcal Ap(t,s)
+
\mathcal L(t+s)p(t,s)
+
\mathcal R\bigl(p(t,s)\bigr)
+
Bz(t,s),
&
s\in(0,D),
\\[1mm]
p(t,0)=w(t).
\end{cases}
\label{eq:nonlinear-predictor-equation}
\end{equation}
Equivalently,
\begin{equation}
\begin{aligned}
p(t,s)
&=
e^{s\mathcal A}w(t)
+
\int_0^s
e^{(s-\sigma)\mathcal A}
\left[
\mathcal L(t+\sigma)p(t,\sigma)
+
\mathcal R\bigl(p(t,\sigma)\bigr)
+
Bz(t,\sigma)
\right]d\sigma,
\end{aligned}
\label{eq:nonlinear-predictor-mild-form}
\end{equation}
for \(s\in[0,D]\). Thus, the predictor depends only on the current
augmented state \((w(t),z(t,\cdot))\) and the prescribed reference segment
\(u_{\mathrm{ref}}^{\,t}\).

\begin{proposition}
\label{prop:predictor-wellposedness}
For every
\[
t\geq0,
\qquad
w(t)\in H,
\qquad
z(t,\cdot)\in\mathcal Z,
\]
equation~\eqref{eq:nonlinear-predictor-equation} admits a unique solution
\[
p(t,\cdot)
\in
C([0,D];H)\cap L^2((0,D);V),
\qquad
\partial_sp(t,\cdot)\in L^2((0,D);V').
\]
There exists \(C_{\mathrm p,D}>0\), depending only on \(D\), \(L\),
\(\|B\|_{\mathcal L(\mathbb R^m,H)}\), and \(M_{\mathscr A}\), such that
\begin{equation}
\sup_{s\in[0,D]}\|p(t,s)\|_H^2
+
\int_0^D\|p(t,s)\|_V^2\,ds
\leq
C_{\mathrm p,D}
\left(
\|w(t)\|_H^2
+
\|z(t,\cdot)\|_{\mathcal Z}^2
\right).
\label{eq:predictor-global-bound}
\end{equation}
The estimate is global in \((w(t),z(t,\cdot))\) and uniform over all
complete reference trajectories in \(\mathscr A\).
\end{proposition}

\begin{proof}
Let \(p_n\) denote the Fourier--Galerkin approximation. Periodicity and \eqref{eq:reference-derivative-uniform-bound} give
\[
\langle\mathcal R(p_n),p_n\rangle_H=0,
\qquad
\langle\mathcal L(t+s)p_n,p_n\rangle_H
\leq
C_{\mathscr A}\|p_n\|_H^2.
\]
Hence,
\[
\frac12\frac{d}{ds}\|p_n\|_H^2
+
\|(p_n)_{xx}\|_{L^2}^2
-
\|(p_n)_x\|_{L^2}^2
\leq
C_{\mathscr A}\|p_n\|_H^2
+
\langle Bz,p_n\rangle_H.
\]
Interpolation and Young's inequality imply
\begin{equation}
\frac{d}{ds}\|p_n\|_H^2
+
\|(p_n)_{xx}\|_{L^2}^2
\leq
C\|p_n\|_H^2
+
C\|z(t,s)\|_{\mathbb R^m}^2,
\label{eq:predictor-energy-inequality}
\end{equation}
where \(C\) is independent of \(t\), \(n\), and the reference trajectory.
Gronwall's inequality and integration over \((0,D)\) give the bound
\eqref{eq:predictor-global-bound} at the Galerkin level.

The estimates
\[
\|\mathcal L(t+s)p_n\|_{V'}
\leq
CM_{\mathscr A}\|p_n\|_H,
\qquad
\|\mathcal R(p_n)\|_{V'}
\leq
C\|p_n\|_H^2
\]
also bound \((\partial_sp_n)\) in \(L^2((0,D);V')\).
The Aubin--Lions lemma permits passage to the nonlinear term and gives the asserted solution and estimate. Uniqueness follows from the corresponding difference estimate and Gronwall's inequality.
\end{proof}

Define the causal predictor feedback by
\begin{equation}
U(t)
=
\kappa\bigl(p(t,D)\bigr)
=
-\mu B^\ast p(t,D),
\qquad
t\geq\tau.
\label{eq:predictor-feedback-law}
\end{equation}
The inflow boundary condition of the delay line consequently becomes
\begin{equation}
z(t,D)=-\mu B^\ast p(t,D).
\label{eq:closed-loop-delay-boundary}
\end{equation}

\begin{proposition}
\label{prop:predictor-identity}
Let \((w,z)\) be a closed-loop mild solution on its maximal existence interval \([\tau,T_{\max})\). For every
\(t\in[\tau,T_{\max})\),
\begin{equation}
p(t,s)=w(t+s),
\qquad
0\leq s\leq
\min\{D,T_{\max}-t\}.
\label{eq:predictor-identity}
\end{equation}
In particular, if \(t+D<T_{\max}\), then
\begin{equation}
p(t,D)=w(t+D).
\label{eq:terminal-predictor-identity}
\end{equation}
The identities hold in \(H\), and pointwise for classical solutions.
\end{proposition}

\begin{proof}
Fix \(t\in[\tau,T_{\max})\) and let
\[
D_t:=\min\{D,T_{\max}-t\}.
\]
For \(s\in[0,D_t]\), set
\[
\widehat p(s)=w(t+s).
\]
The characteristic representation of \(z_t=z_s\) gives
\begin{equation}
z(t+s,0)=z(t,s)
\qquad
\text{for almost every }s\in(0,D_t).
\label{eq:transport-shift-identity}
\end{equation}
Applying the mild tracking-error equation between \(t\) and \(t+s\)
and using~\eqref{eq:transport-shift-identity} yields
\[
\begin{aligned}
\widehat p(s)
&=
e^{s\mathcal A}w(t)
+
\int_0^s
e^{(s-\sigma)\mathcal A}
\left[
\mathcal L(t+\sigma)\widehat p(\sigma)
+
\mathcal R\bigl(\widehat p(\sigma)\bigr)
+
Bz(t,\sigma)
\right]d\sigma.
\end{aligned}
\]
Thus, \(\widehat p\) satisfies the predictor equation on \([0,D_t]\)
with \(\widehat p(0)=w(t)\). Uniqueness from
Proposition~\ref{prop:predictor-wellposedness} gives
\[
p(t,s)=\widehat p(s)=w(t+s)
\]
almost everywhere in \(s\). Continuity in \(H\) extends the equality to every \(s\in[0,D_t]\).
\end{proof}

Proposition~\ref{prop:predictor-identity} is valid on the existing
closed-loop interval and therefore does not presuppose global
well-posedness. If \(t+D<T_{\max}\), combining
\eqref{eq:predictor-feedback-law} and
\eqref{eq:terminal-predictor-identity} gives
\begin{equation}
U(t)=-\mu B^\ast w(t+D).
\label{eq:predictor-feedback-future-state-form}
\end{equation}
The relation is an identity along an existing closed-loop trajectory, rather than an operational definition involving a future plant state. Indeed, the implemented feedback remains the causal law \eqref{eq:predictor-feedback-law}, computed from
\(w(t)\), \(z(t,\cdot)\), and \(u_{\mathrm{ref}}^{\,t}\).

Since \(U(t)\) reaches the plant at time \(t+D\),
\begin{equation}
BU(t)
=
-\mu BB^\ast w(t+D)
=
-\mu P_Nw(t+D).
\label{eq:delay-compensation-identity}
\end{equation}
Hence, whenever the command reaches the plant, its action coincides
exactly with the nominal feedback in \eqref{eq:nominal-closed-loop-system}.

\subsection{Predictor Transformation and Target Dynamics}
\label{subsec:predictor-transformation}

Define the predictor transformation
\begin{equation}
r(t,s)
=
z(t,s)-\kappa\bigl(p(t,s)\bigr)
=
z(t,s)+\mu B^\ast p(t,s),
\qquad
s\in[0,D].
\label{eq:predictor-transformation}
\end{equation}
Since \(p(t,0)=w(t)\),
\begin{equation}
z(t,0)
=
r(t,0)+\kappa\bigl(w(t)\bigr).
\label{eq:delayed-input-transformed}
\end{equation}
Moreover, the feedback law \eqref{eq:predictor-feedback-law} and the boundary condition \(z(t,D)=U(t)\) give
\begin{equation}
r(t,D)=0.
\label{eq:transformed-actuator-boundary}
\end{equation}

To construct the inverse transformation, fix \(t\geq0\) and let
\((w,r)\in H\times\mathcal Z\). Define \(p^r(t,\cdot)\) as the unique solution of

\begin{equation}
\begin{cases}
\begin{aligned}
\partial_s p^r(t,s)
={}&
\mathcal A p^r(t,s)
+\mathcal L(t+s)p^r(t,s)
+\mathcal R\bigl(p^r(t,s)\bigr)
\\
&-\mu P_Np^r(t,s)
+Br(t,s),
\qquad s\in(0,D),
\end{aligned}
\\[1mm]
p^r(t,0)=w.
\end{cases}
\label{eq:inverse-predictor-equation}
\end{equation}

This is the perturbed nominal system \eqref{eq:nominal-closed-loop-system}, with prediction variable \(s\) and disturbance \(Br(t,\cdot)\). Proposition \ref{prop:nominal-uniform-ISS} therefore yields
\begin{equation}
\|p^r(t,s)\|_H^2
\leq
e^{-\eta s}\|w\|_H^2
+
\frac{\|B\|^2}{\eta}
\int_0^s
e^{-\eta(s-\sigma)}
\|r(t,\sigma)\|_{\mathbb R^m}^2\,d\sigma.
\label{eq:inverse-predictor-estimate}
\end{equation}
After integration over \((0,D)\),
\begin{equation}
\|p^r(t,\cdot)\|_{L^2((0,D);H)}^2
\leq
\frac{1-e^{-\eta D}}{\eta}\|w\|_H^2
+
\frac{\|B\|^2}{\eta^2}
\|r(t,\cdot)\|_{\mathcal Z}^2.
\label{eq:inverse-predictor-L2-estimate}
\end{equation}
The inverse transformation is consequently
\begin{equation}
z(t,s)
=
r(t,s)+\kappa\bigl(p^r(t,s)\bigr)
=
r(t,s)-\mu B^\ast p^r(t,s).
\label{eq:inverse-predictor-transformation}
\end{equation}

\begin{proposition}
\label{prop:predictor-transformation-invertibility}
For every \(t\geq0\), transformation \eqref{eq:predictor-transformation} is a bijection on \(H\times\mathcal Z\), with inverse \eqref{eq:inverse-predictor-equation}--%
\eqref{eq:inverse-predictor-transformation}. Moreover, there exist
constants \(C_{\mathrm d},C_{\mathrm i}>0\), independent of \(t\) and of the selected complete reference trajectory, such that
\begin{align}
\|w\|_H^2+\|r(t,\cdot)\|_{\mathcal Z}^2
&\leq
C_{\mathrm d}
\bigl(
\|w\|_H^2+\|z(t,\cdot)\|_{\mathcal Z}^2
\bigr),
\label{eq:direct-transformation-bound}
\\
\|w\|_H^2+\|z(t,\cdot)\|_{\mathcal Z}^2
&\leq
C_{\mathrm i}
\bigl(
\|w\|_H^2+\|r(t,\cdot)\|_{\mathcal Z}^2
\bigr).
\label{eq:inverse-transformation-bound}
\end{align}
Thus, the original and transformed state norms are globally and uniformly equivalent.
\end{proposition}

\begin{proof}
Let \((w,z)\in H\times\mathcal Z\), and let \(p\) solve
\eqref{eq:nonlinear-predictor-equation}. Substitution of
\[
z=r+\kappa(p)
\]
into that equation shows that \(p\) satisfies
\eqref{eq:inverse-predictor-equation}. Its uniqueness gives \(p=p^r\);
hence \eqref{eq:inverse-predictor-transformation} recovers \(z\).

Conversely, let \(p^r\) solve
\eqref{eq:inverse-predictor-equation} and define \(z\) by
\eqref{eq:inverse-predictor-transformation}. Since
\[
Bz
=
Br+B\kappa(p^r),
\]
the function \(p^r\) satisfies \eqref{eq:nonlinear-predictor-equation} with input \(z\).
Consequently, applying \eqref{eq:predictor-transformation} recovers \(r\), which proves bijectivity.

Finally,
\[
\|r\|_{\mathcal Z}^2
\leq
2\|z\|_{\mathcal Z}^2
+
2\mu^2\|B^\ast\|^2
\|p\|_{L^2((0,D);H)}^2.
\]
The direct estimate follows from
\eqref{eq:predictor-global-bound}. Similarly,
\[
\|z\|_{\mathcal Z}^2
\leq
2\|r\|_{\mathcal Z}^2
+
2\mu^2\|B^\ast\|^2
\|p^r\|_{L^2((0,D);H)}^2,
\]
and \eqref{eq:inverse-predictor-L2-estimate} gives the inverse estimate.
All constants depend only on the fixed design parameters, \(D\), and
\(M_{\mathscr A}\).
\end{proof}


The predictor transformation maps the delayed closed-loop system into the cascade
\begin{equation}
\begin{cases}
\displaystyle
\dot w
=
\mathcal Aw
+
\mathcal L(t)w
+
\mathcal R(w)
-
\mu P_Nw
+
Br(t,0),
\\[1mm]
\displaystyle
\partial_t r(t,s)
=
\partial_s r(t,s),
&
s\in(0,D),
\\[1mm]
r(t,D)=0.
\end{cases}
\label{eq:predictor-target-system}
\end{equation}
Thus, the tracking component is the nominal closed-loop system driven by the trace \(r(t,0)\), while the transport component is autonomous and vanishes after one delay horizon.

To verify this transformation, first consider compatible regular data and a corresponding classical solution on a common existence interval. For every \(t\in [\tau, T_{max}]\) and \(s\in[0,D]\) satisfying
\[
t+s<T_{\max},
\]
the predictor identity and the characteristic representation of the actuator state give
\begin{equation}
p(t,s)=w(t+s),
\qquad
z(t,s)=z(t+s,0).
\label{eq:predictor-characteristic-identities}
\end{equation}
Consequently,
\[
r(t,s)
=
z(t+s,0)-\kappa\bigl(w(t+s)\bigr)
=
r(t+s,0),
\]
and hence
\begin{equation}
\partial_t r(t,s)=\partial_s r(t,s).
\label{eq:transformed-transport-equation}
\end{equation}
Moreover, the feedback law gives
\[
r(t,D)
=
z(t,D)-\kappa\bigl(p(t,D)\bigr)
=
U(t)-U(t)
=
0.
\]
Since \(p(t,0)=w(t)\), one also has
\[
z(t,0)=r(t,0)+\kappa\bigl(w(t)\bigr).
\]
Substitution into the tracking equation yields
\[
\dot w
=
\mathcal Aw
+
\mathcal L(t)w
+
\mathcal R(w)
+
B\kappa(w)
+
Br(t,0).
\]
The identity
\[
B\kappa(w)=-\mu P_Nw
\]
then yields the first equation of
\eqref{eq:predictor-target-system}.

For mild solutions, these identities are understood through the
characteristic representation of the transport equation and the
variation-of-constants formula for the tracking equation.

\begin{proposition}
\label{prop:target-wellposedness}
For every \(\tau\geq0\) and every initial state
\[
(w_\tau,r_\tau)\in\mathcal X,
\qquad
\mathcal X=H\times\mathcal Z,
\]
system~\eqref{eq:predictor-target-system} admits a unique global mild solution
\[
(w,r)\in C([\tau,\infty);\mathcal X),
\]
such that
\begin{equation}
w\in L_{\mathrm{loc}}^2([\tau,\infty);V),
\qquad
\dot w\in L_{\mathrm{loc}}^2([\tau,\infty);V').
\label{eq:target-additional-regularity}
\end{equation}
The transport component is given by
\begin{equation}
r(t,s)
=
\begin{cases}
r_\tau\bigl(s+t-\tau\bigr),
&
s+t-\tau<D,
\\[1mm]
0,
&
s+t-\tau\geq D,
\end{cases}
\label{eq:target-transport-explicit-solution}
\end{equation}
for \(t\geq\tau\) and almost every \(s\in(0,D)\). In particular,
\begin{equation}
r(t,\cdot)=0
\quad\text{in }\mathcal Z,
\qquad
t\geq\tau+D.
\label{eq:transport-finite-time-extinction}
\end{equation}
\end{proposition}

\begin{proof}
The transport equation in
\eqref{eq:predictor-target-system}, together with the inflow boundary condition \(r(t,D)=0\), generates the left-shift semigroup on \(\mathcal Z\). Its characteristic representation is
\eqref{eq:target-transport-explicit-solution}. Consequently,
\[
r\in C([\tau,\infty);\mathcal Z)
\]
and
\begin{equation}
\|r(t,\cdot)\|_{\mathcal Z}^2
=
\begin{cases}
\displaystyle
\int_{t-\tau}^{D}
\|r_\tau(\sigma)\|_{\mathbb R^m}^2\,d\sigma,
&
0\leq t-\tau<D,
\\[3mm]
0,
&
t-\tau\geq D.
\end{cases}
\label{eq:target-transport-energy}
\end{equation}
Its outgoing trace is
\begin{equation}
r(t,0)
=
\begin{cases}
r_\tau(t-\tau),
&
0<t-\tau<D,
\\[1mm]
0,
&
t-\tau\geq D,
\end{cases}
\label{eq:target-outflow-trace}
\end{equation}
for almost every \(t>\tau\). Therefore,
\[
r(\cdot,0)\in L^2((\tau,\tau+D);\mathbb R^m)
\]
and
\begin{equation}
\int_\tau^{\tau+D}
\|r(t,0)\|_{\mathbb R^m}^2\,dt
=
\|r_\tau\|_{\mathcal Z}^2.
\label{eq:target-outflow-trace-identity}
\end{equation}

It follows that
\[
Br(\cdot,0)\in L^2((\tau,\tau+D);H),
\]
and this forcing vanishes for almost every \(
t\geq\tau+D,\).
The global well-posedness result for the nominal tracking equation, applied to the finite-duration forcing \(Br(t,0)\), gives a unique global solution satisfying
\[
w\in C([\tau,\infty);H)
\cap
L_{\mathrm{loc}}^2([\tau,\infty);V),
\qquad
\dot w\in L_{\mathrm{loc}}^2([\tau,\infty);V').
\]
The transport component then completes the global well-posedness and regularity proof for the coupled system.
\end{proof}

\begin{proposition}
\label{prop:closed-loop-target-equivalence}
On every common existence interval, the predictor transformation
\eqref{eq:predictor-transformation} establishes a one-to-one correspondence between mild solutions of the delayed closed-loop system and mild solutions of the target system~\eqref{eq:predictor-target-system}. Its inverse is given
by
\eqref{eq:inverse-predictor-equation}--%
\eqref{eq:inverse-predictor-transformation}.
\end{proposition}

\begin{proof}
Consider first compatible regular initial data. The characteristic argument above shows that the direct predictor transformation maps every classical solution of the delayed closed-loop system into a classical solution of \eqref{eq:predictor-target-system}.

Conversely, let \((w,r)\) be a classical solution of the target system.
For every \(t\) in its existence interval, let \(p^r(t,\cdot)\) be the solution of \eqref{eq:inverse-predictor-equation}, and define
\begin{equation}
z(t,s)
=
r(t,s)+\kappa\bigl(p^r(t,s)\bigr).
\label{eq:inverse-transformation-in-equivalence-proof}
\end{equation}
Since \(p^r(t,0)=w(t)\), one has
\[
z(t,0)
=
r(t,0)+\kappa\bigl(w(t)\bigr).
\]
The first equation of \eqref{eq:predictor-target-system} then becomes
\[
\dot w
=
\mathcal Aw
+
\mathcal L(t)w
+
\mathcal R(w)
+
Bz(t,0),
\]
which recovers the original tracking equation.

At \(s=D\), the boundary condition \(r(t,D)=0\) gives
\[
z(t,D)
=
\kappa\bigl(p^r(t,D)\bigr)
=
U(t),
\]
and hence the feedback boundary condition is also recovered. Finally, the characteristic representation of \(r\), together with uniqueness of the inverse predictor equation, implies that \(z\) satisfies
\[
\partial_t z=\partial_s z
\]
in the distributional sense. Therefore, \((w,z)\) solves the delayed
closed-loop system.

For general initial data in \(\mathcal X\), approximate the data by compatible regular data. The uniform bounds
\eqref{eq:direct-transformation-bound}--\eqref{eq:inverse-transformation-bound}, together with continuous dependence of both systems on their initial states, permit passage to the limit. The preceding argument establishes the one-to-one correspondence between the mild solutions of the original and transformed systems.
\end{proof}

After the finite-time extinction of the transport component,
\[
r(t,\cdot)=0,
\qquad
t\geq \tau+D,
\]
the tracking equation reduces exactly to
\begin{equation}
\dot w
=
\mathcal Aw
+
\mathcal L(t)w
+
\mathcal R(w)
-
\mu P_Nw,
\qquad
t\geq \tau+D.
\label{eq:post-delay-nominal-dynamics}
\end{equation}
Hence, the predictor transformation converts the delayed closed-loop system into the nominal tracking dynamics driven, only during the initial delay interval, by a transport component that vanishes in finite time.

\subsection{Exponential Stability of the Delayed Closed-Loop System}
\label{subsec:closed-loop-stability}

Let \(\tau\geq0\) and
\[
(w_\tau,z_\tau)\in\mathcal X.
\]
Denote by \(p_\tau\) the predictor state associated with
\((w_\tau,z_\tau)\) at time \(\tau\), and set
\begin{equation}
r_\tau(s)
=
z_\tau(s)-\kappa\bigl(p_\tau(s)\bigr),
\qquad
\text{for a.e. }s\in(0,D).
\label{eq:initial-transformed-state}
\end{equation}
The direct transformation estimate
\eqref{eq:direct-transformation-bound} gives
\begin{equation}
\|w_\tau\|_H^2+\|r_\tau\|_{\mathcal Z}^2
\leq
C_{\mathrm d}
\left(
\|w_\tau\|_H^2+\|z_\tau\|_{\mathcal Z}^2
\right).
\label{eq:initial-transformed-state-bound}
\end{equation}

\begin{theorem}[Uniform exponential stability]
\label{thm:delayed-closed-loop-exponential-stability}
Let \(N\in\mathbb N\) and \(\mu>0\) satisfy
\eqref{eq:nominal-spectral-condition}, and let
\[
\eta
=
\underline a_{N,\mu}-C_{\mathscr A}
>
0.
\]
Then, for every \(\tau\geq0\), every complete reference trajectory in \(\mathscr A\), and every \((w_\tau,z_\tau)\in\mathcal X\), the delayed closed-loop system governed by \eqref{eq:predictor-feedback-law} admits a unique global mild solution
\[
(w,z)\in C([\tau,\infty);\mathcal X),
\]
such that
\begin{equation}
w\in L_{\mathrm{loc}}^2([\tau,\infty);V),
\qquad
\dot w\in L_{\mathrm{loc}}^2([\tau,\infty);V').
\label{eq:delayed-solution-additional-regularity}
\end{equation}
Moreover,
\begin{equation}
\|(w(t),z(t,\cdot))\|_{\mathcal X}
\leq
M_D e^{-\frac{\eta}{2}(t-\tau)}
\|(w_\tau,z_\tau)\|_{\mathcal X},
\qquad
t\geq\tau,
\label{eq:delayed-closed-loop-exponential-estimate}
\end{equation}
where \(M_D\geq1\) is independent of \(\tau\), the selected reference trajectory, and the initial data. Hence, the origin is globally and uniformly exponentially stable in \(\mathcal X\).
\end{theorem}

\begin{proof}
Proposition~\ref{prop:closed-loop-target-equivalence} and the global
well-posedness of the target system yield a unique global mild solution with the stated regularity.

Set
\[
\theta=t-\tau.
\]
The characteristic representation of the transport component gives
\begin{equation}
r(t,s)
=
\begin{cases}
r_\tau(s+\theta),
&
s+\theta<D,
\\
0,
&
s+\theta\geq D,
\end{cases}
\label{eq:shifted-target-transport-solution}
\end{equation}
for almost every \(s\in(0,D)\). Consequently,
\begin{equation}
r(t,\cdot)=0
\quad\text{in }\mathcal Z,
\qquad
t\geq\tau+D,
\label{eq:shifted-finite-time-extinction}
\end{equation}
and
\begin{equation}
\|r(t,\cdot)\|_{\mathcal Z}^2
\leq
e^{\eta D}e^{-\eta\theta}
\|r_\tau\|_{\mathcal Z}^2.
\label{eq:shifted-transport-exponential-envelope}
\end{equation}

Moreover, the outgoing trace satisfies
\[
r(t,0)
=
\begin{cases}
r_\tau(\theta),
&
0<\theta<D,
\\
0,
&
\theta\geq D,
\end{cases}
\]
for almost every \(t>\tau\). Applying Proposition~\ref{prop:nominal-uniform-ISS} to the tracking component with \(d(t)=Br(t,0)\) gives
\begin{align*}
\|w(t)\|_H^2
&\leq
e^{-\eta\theta}\|w_\tau\|_H^2
+
\frac{\|B\|_{\mathcal L(\mathbb R^m,H)}^2}{\eta}
\int_0^{\min\{\theta,D\}}
e^{-\eta(\theta-\rho)}
\|r_\tau(\rho)\|_{\mathbb R^m}^2\,d\rho
\nonumber\\
&\leq
e^{-\eta\theta}
\left(
\|w_\tau\|_H^2
+
\frac{\|B\|_{\mathcal L(\mathbb R^m,H)}^2}{\eta}
e^{\eta D}\|r_\tau\|_{\mathcal Z}^2
\right).
\label{eq:target-tracking-decay}
\end{align*}
Combining this estimate with
\eqref{eq:shifted-transport-exponential-envelope} yields
\begin{equation}
\|w(t)\|_H^2+\|r(t,\cdot)\|_{\mathcal Z}^2
\leq
C_D e^{-\eta(t-\tau)}
\left(
\|w_\tau\|_H^2+\|r_\tau\|_{\mathcal Z}^2
\right),
\label{eq:target-cascade-exponential-estimate}
\end{equation}
where
\begin{equation}
C_D
=
\max
\left\{
1,\,
e^{\eta D}
\left(
1+
\frac{\|B\|_{\mathcal L(\mathbb R^m,H)}^2}{\eta}
\right)
\right\}.
\label{eq:target-cascade-constant}
\end{equation}

The inverse transformation estimate
\eqref{eq:inverse-transformation-bound} and
\eqref{eq:initial-transformed-state-bound} now imply
\begin{align}
\|w(t)\|_H^2+\|z(t,\cdot)\|_{\mathcal Z}^2
&\leq
C_{\mathrm i}C_{\mathrm d}C_D
e^{-\eta(t-\tau)}
\left(
\|w_\tau\|_H^2+\|z_\tau\|_{\mathcal Z}^2
\right).
\label{eq:original-delayed-state-decay}
\end{align}
Taking square roots proves
\eqref{eq:delayed-closed-loop-exponential-estimate} with
\begin{equation}
M_D
=
\sqrt{C_{\mathrm i}C_{\mathrm d}C_D}.
\label{eq:delayed-exponential-constant}
\end{equation}
The uniformity follows from that of
\(C_{\mathrm d}\), \(C_{\mathrm i}\), and the nominal stability
estimate.
\end{proof}

\begin{remark}
\label{rem:classical-initial-compatibility}
No pointwise compatibility condition is required for mild initial data in \(\mathcal X\). For classical solutions, if
\[
z_\tau\in H^1((0,D);\mathbb R^m),
\]
compatibility at the inflow corner \((\tau,D)\) additionally requires
\begin{equation}
z_\tau(D)
=
\kappa\bigl(p_\tau(D)\bigr),
\qquad
\text{equivalently}\qquad
r_\tau(D)=0.
\label{eq:classical-delay-compatibility}
\end{equation}
\end{remark}

\section{Numerical Illustration}
\label{sec:numerical-illustration}

Numerical experiments illustrate the loss of tracking induced by the input delay and the recovery achieved through predictor compensation. The terminal predictor relation and the finite-time extinction of the transformed actuator state are also assessed numerically. The observed tracking decay and actuator behavior are consistent with the mechanisms established in Theorem~\ref{thm:delayed-closed-loop-exponential-stability}.

\subsection{Numerical Setting}
\label{subsec:numerical-setting}

The reference and controlled Kuramoto--Sivashinsky equations are simulated on
\[
\mathbb T_L=\mathbb R/(L\mathbb Z),
\qquad L=22,
\]
over \(0\leq t\leq T\), with
\begin{equation}
T=40,\qquad D=1,\qquad N=2,\qquad \mu=8.
\label{eq:numerical-parameters}
\end{equation}
Thus, the feedback acts on the first two nonzero Fourier wave numbers.

A Fourier pseudospectral discretization with \(N_x=128\) equidistant nodes is employed. Spatial derivatives are evaluated in Fourier space, whereas the quadratic term is computed in physical space and filtered according to the \(2/3\) de-aliasing rule. Time integration is performed by the fourth-order exponential time-differencing Runge--Kutta method with \(\Delta t=10^{-2}\). The zero Fourier coefficient is reset after every time update to eliminate roundoff-induced drift from the zero-mean phase space.

Let \(q=2\pi/L\). The initial reference state and tracking error are prescribed by
\begin{equation}
\begin{aligned}
u_{\mathrm{ref}}(0,x)
&=
0.55\cos(qx)
-0.30\sin(2qx)
+0.16\cos(3qx+0.2),
\\
w_0(x)
&=
0.22\cos(qx+0.4)
-0.11\sin(2qx-0.3)
+0.05\cos(4qx),
\end{aligned}
\label{eq:numerical-initial-data}
\end{equation}
and \(u(0,\cdot)=u_{\mathrm{ref}}(0,\cdot)+w_0\).
Both initial states have zero spatial mean, while \(w_0\not\equiv0\) introduces a nontrivial tracking mismatch. The reference solution is generated once by integrating the uncontrolled KS equation and is then used in all three configurations.

For \(N=2\), the actuator space has dimension \(m=4\), and
\[
B\xi
=
\sqrt{\frac{2}{L}}
\sum_{j=1}^{2}
\left[
\xi_{2j-1}\cos(jqx)
+
\xi_{2j}\sin(jqx)
\right],
\qquad \xi\in\mathbb R^4.
\]
The actuator profiles form an orthonormal family, so that \(BB^\ast=P_2\).
Consequently, the feedback map \(\kappa(v)=-\mu B^\ast v\) realizes
\(B\kappa(v)=-\mu P_2v\), in agreement with the low-mode damping term used in the stability analysis.

The delay interval is discretized with
\[
\Delta s=\Delta t=10^{-2},
\qquad
N_D=\frac{D}{\Delta s}=100.
\]
The relation \(D=N_D\Delta t\) allows the transport state to be advanced by a discrete characteristic shift without interpolation. At each time level, the predictor equation is integrated on \([0,D]\) from \(p(t,0)=w(t)\), and the command is evaluated as
\begin{equation}
U(t)=-\mu B^\ast p(t,D).
\label{eq:numerical-predictor-control}
\end{equation}
The control prehistory is set to zero on \([-D,0)\), equivalently
\(z(0,\cdot)=0\).

Three configurations are compared: the uncontrolled dynamics, the
uncompensated delayed feedback \(U(t)=-\mu B^\ast w(t)\), whose command reaches the plant one delay interval later, and the predictor-compensated feedback \eqref{eq:numerical-predictor-control}. All three simulations use the same spatial discretization, time step, initial data, and feedback parameters.

\subsection{Tracking and Predictor Performance}
\label{subsec:numerical-performance}

Define the normalized tracking error by
\begin{equation}
E(t)=\frac{\|w(t)\|_{L^2}}{\|w_0\|_{L^2}}.
\label{eq:normalized-tracking-error}
\end{equation}
Figure~\ref{fig:tracking-and-control}\subref{fig:tracking-comparison} shows three qualitatively different tracking regimes. At \(T=40\),
\begin{equation}
\begin{aligned}
E_{\mathrm{ol}}(T)&=1.1162\times10^{1},&
E_{\mathrm{del}}(T)&=4.3974\times10^{2},&
E_{\mathrm{pred}}(T)&=2.1111\times10^{-2}.
\end{aligned}
\label{eq:numerical-final-errors}
\end{equation}

The open-loop increase reflects the sensitivity of the KS dynamics to the initial mismatch: the reference and controlled solutions progressively lose phase coherence, although each trajectory may remain bounded. For the uncompensated delayed feedback, the command is computed from an error that no longer represents the state encountered when the command reaches the plant.
The resulting temporal misalignment can reverse the intended low-mode damping effect and strongly amplify the tracking error over the simulated interval.

Predictor compensation restores the alignment between the information used to construct the command and the future state on which that command acts. The final error is approximately \(2.1\%\) of its initial value. The predictor neither removes nor shortens the physical actuator delay; it anticipates the evolution during signal propagation and recovers the dissipative effect of the
nominal low-mode feedback.

A least-squares fit of \(\log E_{\mathrm{pred}}(t)\) over
\(D+1\leq t\leq 0.65T\) gives
\begin{equation}
\widehat\lambda_{\mathrm{pred}}=0.1198.
\label{eq:empirical-decay-rate}
\end{equation}
The resulting trajectory-dependent, finite-time estimate is distinct from the uniform analytical decay rate established in
Theorem~\ref{thm:delayed-closed-loop-exponential-stability}.

Figure~\ref{fig:tracking-and-control}\subref{fig:control-input} displays the norms of the generated command and of the delayed command reaching the plant.
Their shift by \(D=1\) makes the propagation time through the actuator explicit. The large initial command results from the prescribed mismatch and the initially empty delay line associated with the zero control prehistory.
Its subsequent decrease indicates that progressively less corrective effort is required through the actuated modes as synchronization develops. The remaining oscillations are consistent with tracking a time-dependent trajectory: the controller continues to correct phase and amplitude discrepancies rather than maintaining a stationary state.

\begin{figure}[!t]
\centering
\subfloat[Normalized tracking errors.
\label{fig:tracking-comparison}]
{\includegraphics[width=0.48\textwidth]
{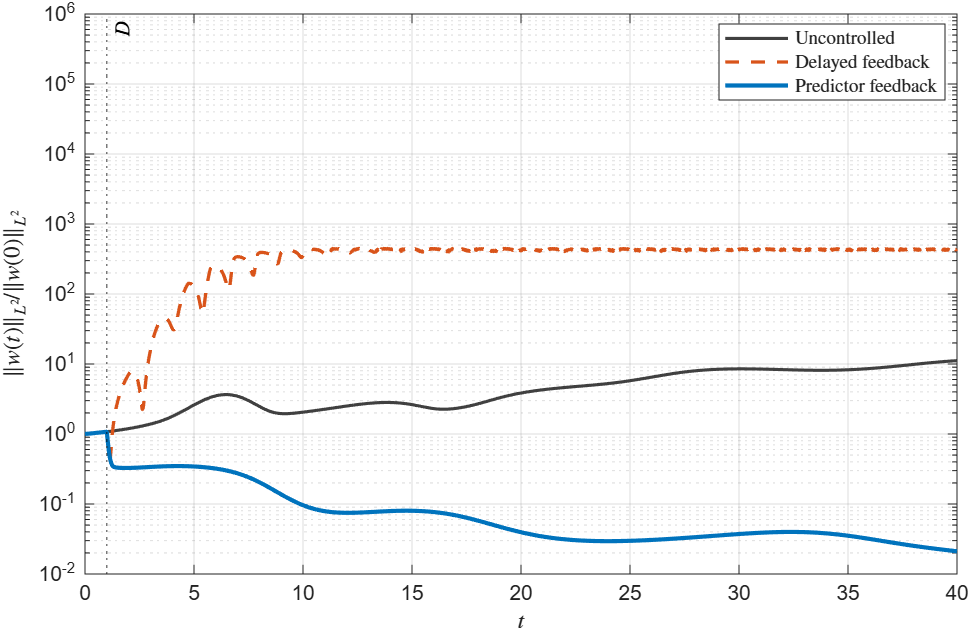}}
\hfill
\subfloat[Generated and applied control norms.
\label{fig:control-input}]
{\includegraphics[width=0.48\textwidth]
{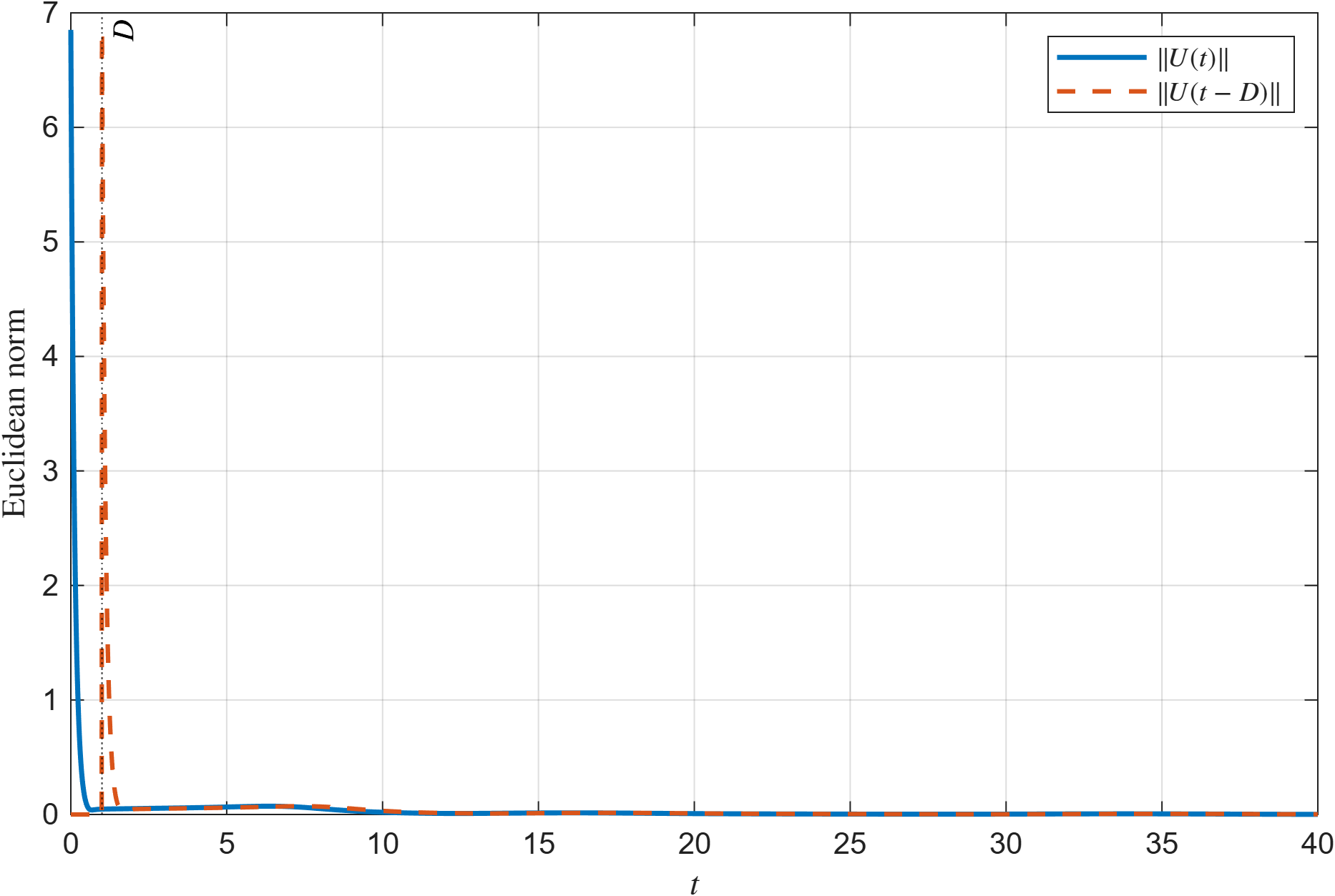}}
\caption{Tracking errors and control effort for the three configurations.}
\label{fig:tracking-and-control}
\end{figure}

The space--time field in
Fig.~\ref{fig:error-and-diagnostics}\subref{fig:error-field} shows that error attenuation is not pointwise monotone. During the transient, spatial error structures are transported, deformed, and redistributed by the nonlinear convective dynamics before being dissipated. Such transient dynamics reflect the coexistence of long-wave instability, nonlinear mode coupling, and fourth-order dissipation in the KS equation.

The feedback directly damps only the first two nonzero wave numbers. The remaining components are attenuated through the combined effect of nonlinear mode coupling, direct low-mode damping, and intrinsic fourth-order dissipation at high wave numbers. The decay over the whole domain therefore represents
synchronization of two evolving spatial patterns rather than relaxation toward a fixed profile.

The delay-compensation mechanism is assessed through the terminal predictor residual. At each recorded time, the predictor is recomputed using \(\Delta s_{\mathrm{aux}}=\Delta s/2\). Denoting this auxiliary predictor by \(p_{\mathrm{aux}}\), define
\begin{equation}
\varepsilon_p(t)
=
\left\|
p_{\mathrm{aux}}(t,D)-w_{\mathrm{pred}}(t+D)
\right\|_{L^2},
\qquad 0\leq t\leq T-D.
\label{eq:predictor-consistency-residual}
\end{equation}
The computation yields
\begin{equation}
\max_{0\leq t\leq T-D}\varepsilon_p(t)
=
7.8904\times10^{-10}.
\label{eq:max-predictor-residual}
\end{equation}
The two terms entering the predictor residual are obtained through separate numerical integrations. Hence, the result measures the numerical consistency of the terminal relation \(p(t,D)=w_{\mathrm{pred}}(t+D)\), rather than the
satisfaction of an imposed algebraic identity. Its small magnitude indicates that the predictor reconstructs, to numerical accuracy, the tracking error present when the command generated at time \(t\) reaches the plant.

The transformed actuator state is reconstructed as
\[
r_{\mathrm{num}}(t,s)
=
z_{\mathrm{num}}(t,s)
-
\kappa\bigl(p_{\mathrm{num}}(t,s)\bigr).
\]
Its norm satisfies
\begin{equation}
\max_{D+\Delta t\leq t\leq T}
\|r_{\mathrm{num}}(t,\cdot)\|_{\mathcal Z}
=
8.4001\times10^{-14}.
\label{eq:numerical-extinction}
\end{equation}
After one delay horizon, the transformed state is therefore reduced to roundoff level, consistently with
\eqref{eq:transport-finite-time-extinction}. Physically, the initial content of the actuator line is transported toward the plant and leaves the delay channel after the travel time \(D\). The predictor-compatible boundary command then determines the actuator state, leading to the numerical extinction of the transformed variable. Both diagnostics appear in Fig.~\ref{fig:error-and-diagnostics}\subref{fig:predictor-checks}.

\begin{figure}[!t]
\centering
\subfloat[Predictor-controlled tracking error.
\label{fig:error-field}]
{\includegraphics[width=0.48\textwidth]
{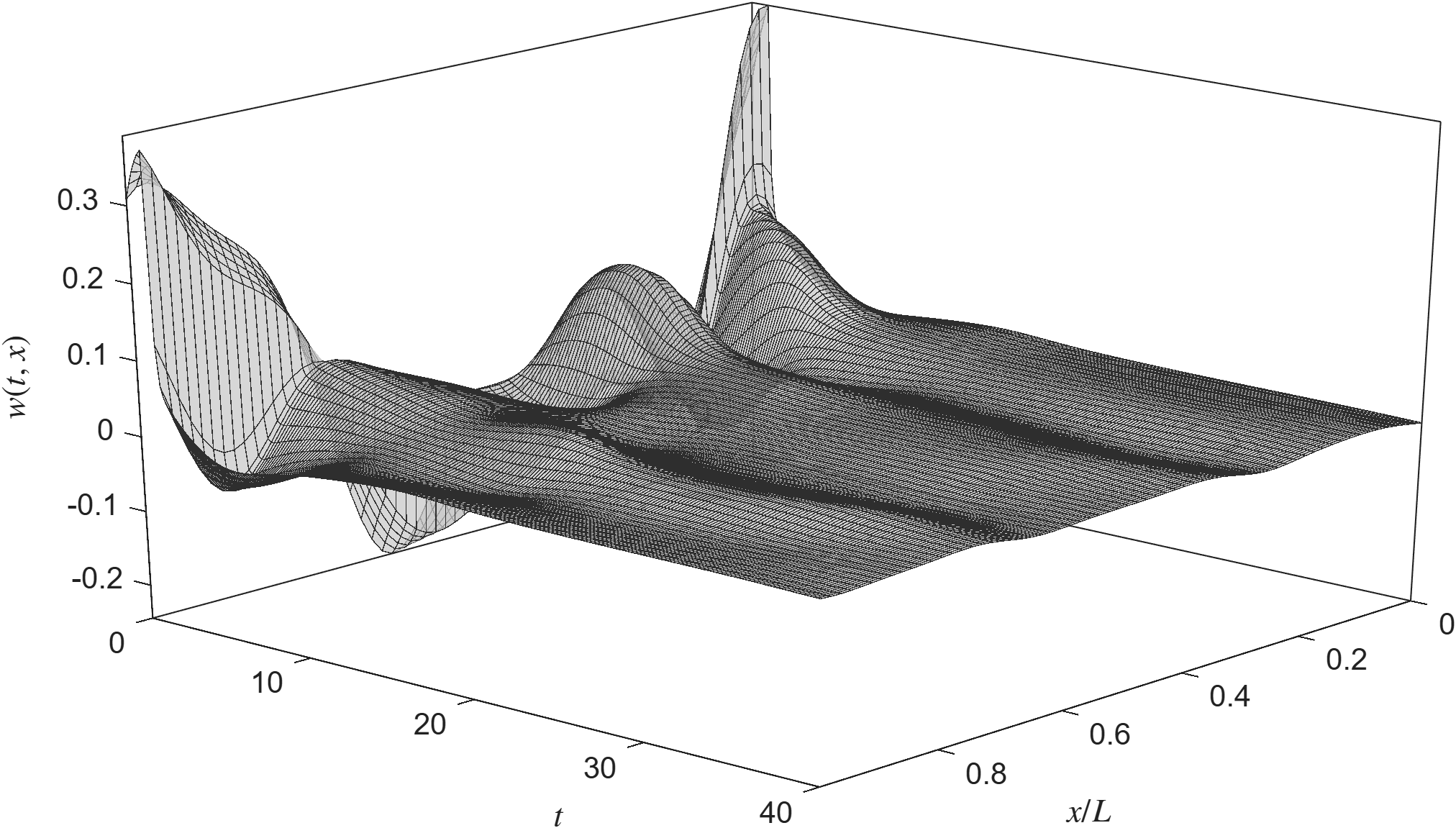}}
\hfill
\subfloat[Predictor and transport diagnostics.
\label{fig:predictor-checks}]
{\includegraphics[width=0.48\textwidth]
{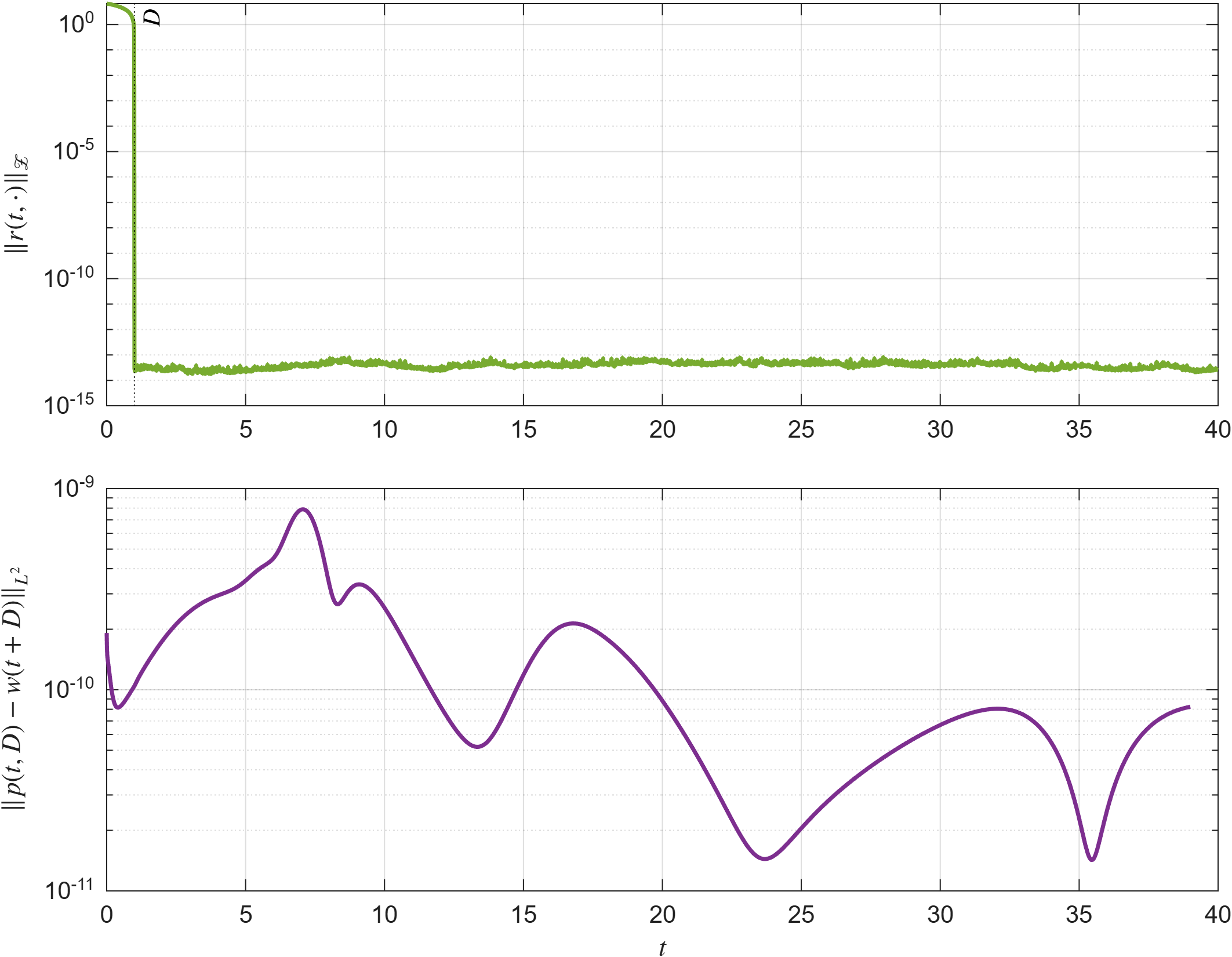}}
\caption{Space--time tracking error and diagnostics associated with the
predictor transformation.}
\label{fig:error-and-diagnostics}
\end{figure}

Figure~\ref{fig:state-comparison} represents the synchronization process in physical space. Initially, the reference and controlled profiles differ in phase and amplitude. During the intermediate window, their dominant large-scale structures become aligned through the direct damping of the first two Fourier wave numbers. Residual discrepancies are subsequently attenuated through nonlinear mode coupling and the intrinsic dissipation of the higher
wave numbers.

Over the final window, the two profiles are nearly indistinguishable at the plotted scale, while their peaks and troughs continue to move and deform. The controller therefore preserves the intrinsic spatiotemporal dynamics instead of driving the system toward an artificial equilibrium. Its physical role is
to synchronize the phase and amplitude of the controlled pattern with the prescribed nonstationary reference trajectory.

\begin{figure}[!t]
\centering
\subfloat[\(t\in\lbrack 0,2\rbrack\).
\label{fig:state-comparison-initial}]
{\includegraphics[width=0.48\textwidth]
{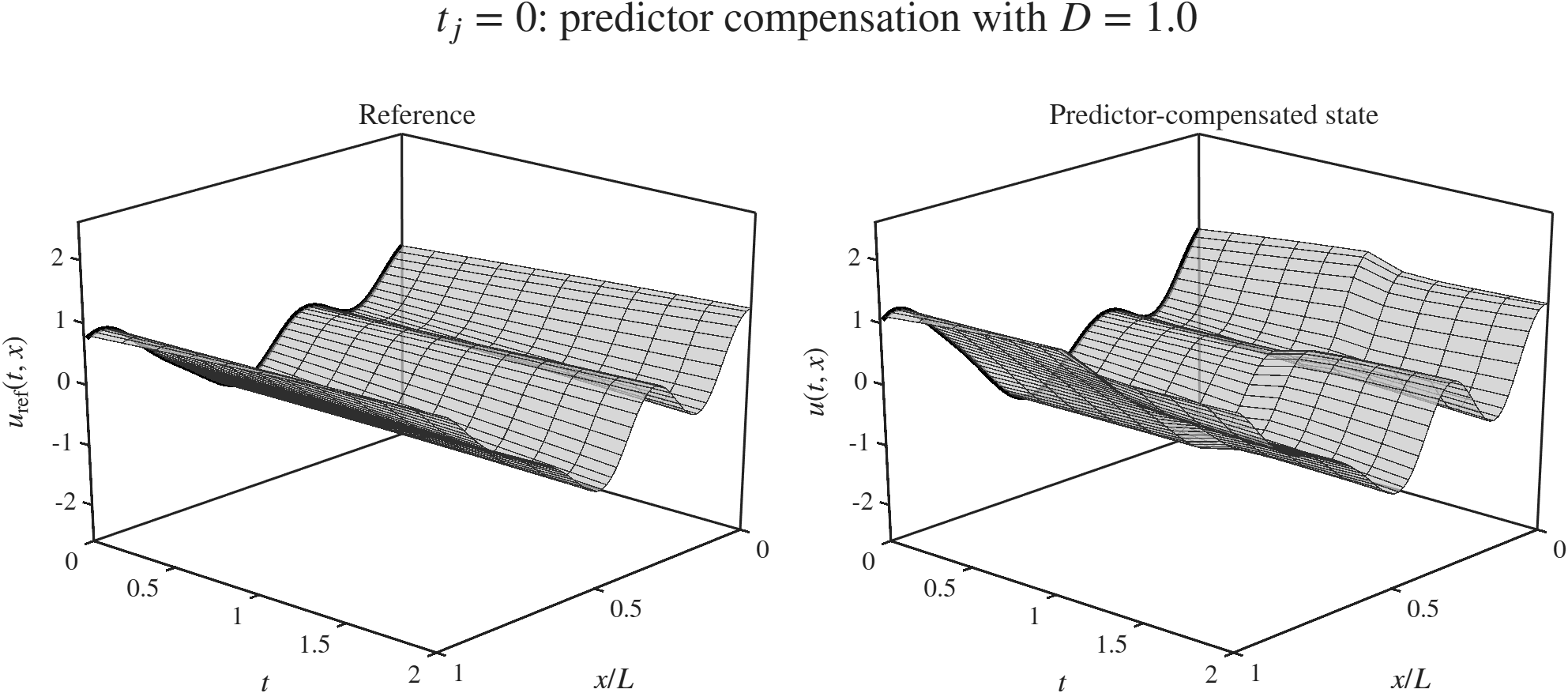}}
\hfill
\subfloat[\(t\in\lbrack 19,21\rbrack\).
\label{fig:state-comparison-intermediate}]
{\includegraphics[width=0.48\textwidth]
{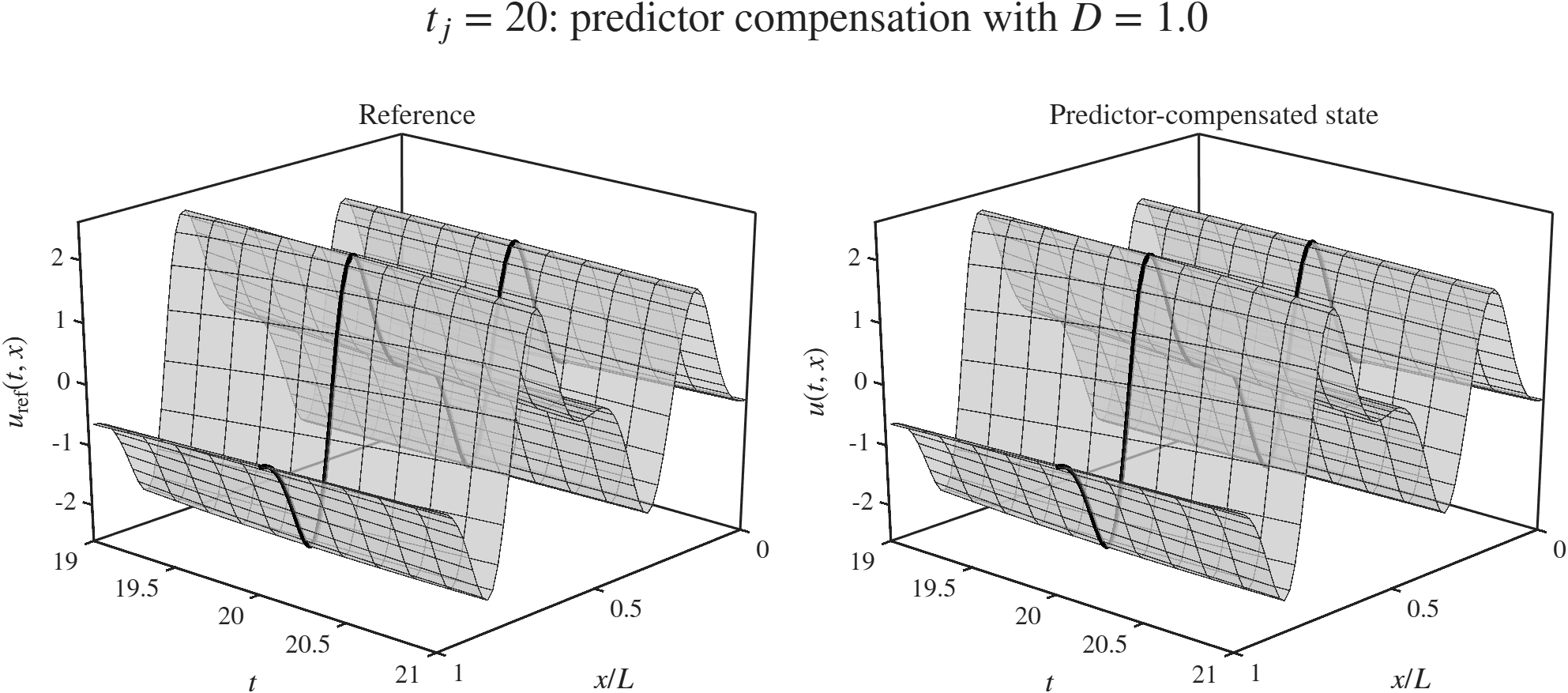}}
\hfill
\subfloat[\(t\in\lbrack 38,40\rbrack\).
\label{fig:state-comparison-final}]
{\includegraphics[width=0.48\textwidth]
{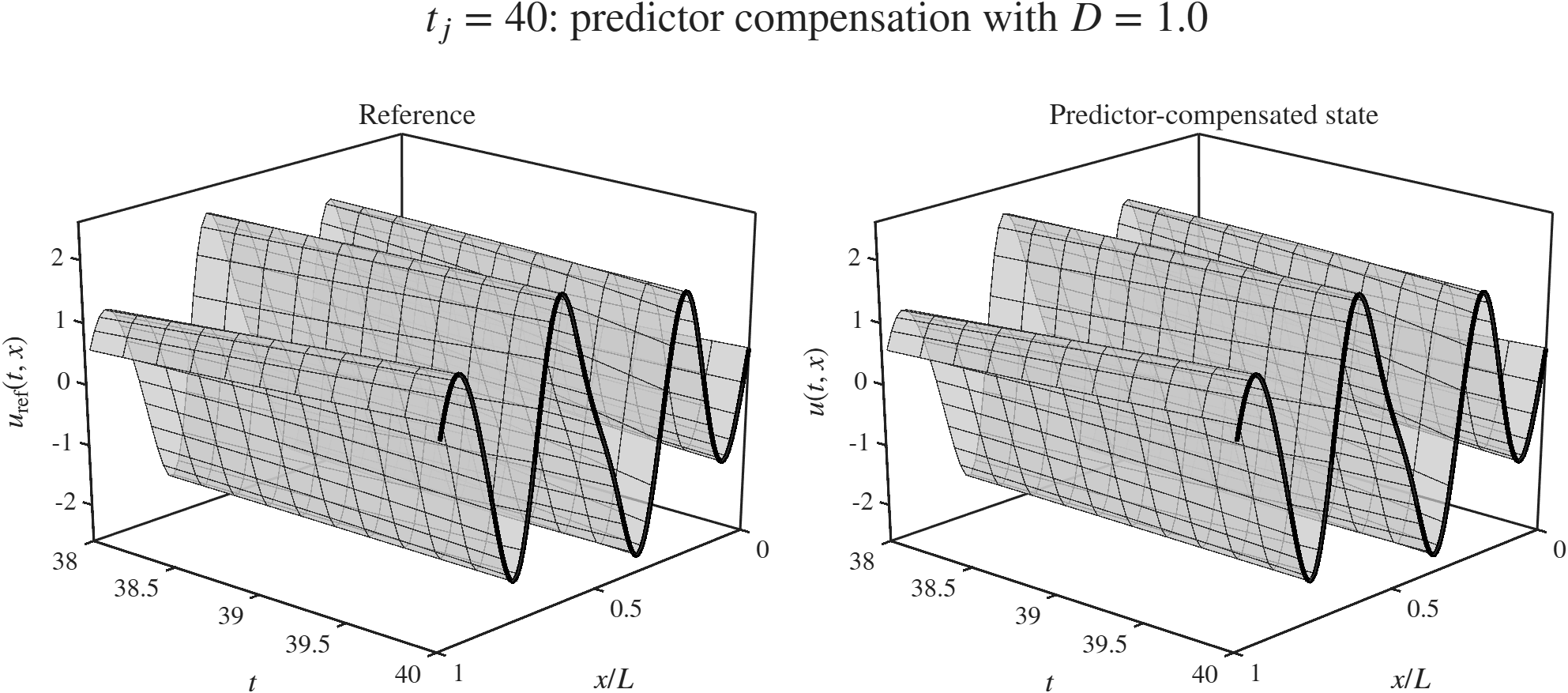}}
\caption{Reference and predictor-controlled states over the initial,
intermediate, and final temporal windows.}
\label{fig:state-comparison}
\end{figure}

Refining the discretization from
\[
(N_x,\Delta t,\Delta s)=(128,10^{-2},10^{-2})
\]
to
\[
(N_x,\Delta t,\Delta s)=(256,5\times10^{-3},5\times10^{-3})
\]
changes \(E_{\mathrm{pred}}(T)\) from \(2.111135\times10^{-2}\) to
\(2.113984\times10^{-2}\), a relative variation of approximately
\(1.35\times10^{-3}\), while the fitted rate changes only from \(0.119830\)
to \(0.119831\). Moreover,
\[
\max_t\left|\overline{u}_{\mathrm{ref}}(t)\right|
=6.5919\times10^{-17},
\qquad
\max_t\left|\overline{u}_{\mathrm{pred}}(t)\right|
=3.9085\times10^{-17},
\]
where the overbar denotes the spatial mean. The tracking behavior and empirical decay rate are therefore insensitive to the considered refinement, while the zero-mean constraint is preserved to machine precision.

\section{Conclusion}
\label{sec:conclusion}

A known constant actuator delay can be accommodated without losing uniform exponential tracking of complete reference trajectories on the Kuramoto--Sivashinsky global attractor. Uniformity with respect to the initial time and the selected reference trajectory is the main consequence of the analysis. In the numerical configurations considered, delayed feedback without compensation significantly degrades the tracking performance, whereas predictor compensation recovers sustained decay of the tracking error. The present design assumes exact knowledge of the constant delay and availability of the full state. Extending the same uniform tracking result to delay mismatch, time-varying or unknown delays, and output-based feedback remains to be investigated.

\end{document}